\documentclass[reqno,10pt]{amsart}
\usepackage{amscd,amssymb,amsmath,color,latexsym,graphics,enumerate, stmaryrd,xspace,verbatim, epic, eepic,url}
\usepackage[all]{xy}
\usepackage{defs}
\usepackage[colorlinks=true]{hyperref}

\def\NN{{\bbN}}
\def\ZZ{{\bbZ}}

\def\sat{{\rm sat}}

\def\.{{,\dots,}}
\def\cO{{\calO}}
\def\cI{{\calI}}
\def\cJ{{\calJ}}

\begin{document}

\author{Michael Temkin}
\title[Height reduction for local uniformization]{Height reduction for local uniformization of varieties and non-archimedean spaces}

\thanks{This work was supported by ERC Consolidator Grant 770922 - BirNonArchGeom. I wish to thank the anonymous referee for many useful comments.}

\address{Einstein Institute of Mathematics\\
               The Hebrew University of Jerusalem\\
                Edmond J. Safra Campus, Giv'at Ram, Jerusalem, 91904, Israel}
\email{michael.temkin@mail.huji.ac.il}

\keywords{Local uniformization, valuation rings, local desingularization, Berkovich analytic spaces.}
\begin{abstract}
It is known since the works of Zariski that the essential difficulty in the local uniformization problem is met already in the case of valuations of height one. In this paper we prove that local uniformization of schemes and non-archimedean analytic spaces rigorously follows from the case of valuations of height one. For non-archimedean spaces this result reduces the problem to studying local structure of smooth Berkovich spaces.
\end{abstract}

\maketitle

\section{Introduction}

\subsection{History and motivation}

\subsubsection{Global conjectures}
A classical conjecture asserts that any function field $K$ over a perfect ground field $k$ possesses a smooth projective model $X$. Later it was strengthened to the desingularization conjecture that any model $X$ possesses a modification $f\:X'\to X$ with a $k$-smooth $X'$, in particular, the family of smooth models is cofinal. Even more generally, modern conjectures expect that any integral qe scheme $X$ possesses a blowing up $X'$ which is regular. Here we use qe instead of quasi-excellent and refer to \cite[\S2.3]{absolutedesing} for definition and basic properties. Some other strengthenings consist in requiring that $f$ is projective or even a blowing up, a divisor $D\subset X$ is also resolved, a canonical resolution is constructed, etc.

There is a very similar set of conjectures for schemes over a valuation ring $R$. The classical semistable reduction conjecture asserts that if $R$ is discretely valued, then up to a finite separable ground field extension any smooth proper variety over $k=\Frac(R)$ possesses a semistable proper model over $R$. Again, a stronger semistable modification conjecture asserts that up to a finite separable ground field extension any model can be blown up to a semistable one, and one can consider more general classes of ground rings, at least the class of all valuation rings, at cost of replacing semistability by log smoothness. In addition, and this is the main goal of this paper, one can study the analytic case, when $k$ is real valued complete and one looks for nice formal models of smooth $k$-analytic spaces. The same conjectures apply also in this case. Since they are rather folklore and hard to find in the literature we provide a precise formulation in Conjecture~\ref{logsmoothconj} followed by a discussion of what is known so far.

\subsubsection{Local conjectures}
Historically, the first relatively general approach to resolution of singularities was due to Zariski and it works from local to global. Zariski understood that a natural localization of the problem is not on a variety (or a scheme) $X$ we want to resolve but on the modifications of $X$ we work with in the resolution process. In other words, one works with the topology generated by modifications $X_i\to X$ and Zariski covers of $X_i$ and it is equivalent to the topology of the associated space $\RZ(X)=\lim_i X_i$ called the Riemann space by Zariski and now called Riemann-Zariski or Zariski-Riemann space. Zariski showed that the points of $\RZ(X)$ are valuations of $k(X)$ centered on $X$ and suggested first to resolve varieties along valuations, and then patch the local solutions after an additional blowing up. The two main results of Zariski were as follows: in characteristic zero resolution of varieties along valuations is possible, see \cite{Zar}, for threefolds this implies global resolution, see \cite{Zariski}. The famous local uniformization conjecture asserts that Zariski's theorem also holds in positive characteristic. Again, there are various natural stronger versions, including what we call the {\em log uniformization conjecture}. It applies to qe schemes, uses blowings up, deals with divisors and controls the exceptional locus, see Conjecture~\ref{logunifconj} and the subsequent discussion. The case of $\dim(X)=3$ is the deepest case established so far, and global resolution of qe threefolds was deduced from it -- both works are due to Cossart-Piltant, see \cite{CP}.

The situation with semistable/log smooth modification conjectures is similar. Their local versions conjecture that such a modification exists locally along a semivaluation on the analytic space $X$. For example, such a semivaluation can be viewed as a point on the associated adic space $X^\ad$. We refer to Conjectures~\ref{locunifconj} and \ref{locunifcor} for precise formulations and a detailed discussion. Again, the case when the residue field $\tilk$ is of characteristic zero is known (though the non-discrete case is very recent), while in general we only know this in the lowest possible dimension -- in the case of curves. Moreover, similarly to Zariski's approach one can first establish local uniformization and then deduce global modification results; this is the main strategy of \cite{temst}.  In view of the analogy with the case of schemes it is natural to expect that the two-dimensional case should be within reach (though very difficult) of the concurrent methods, and it should imply global log smooth modification for surfaces. The author plans to pursue this direction in future works, and this paper will serve as a starting point. Our main goal here is to reduce local uniformization of semivaluations on $k$-analytic spaces to the case of usual analytic points (rather than adic ones), see Theorem~\ref{mainth}. The local situation at such points will be studied in further works by analytic methods.

\subsection{Main results}

\subsubsection{Induction on height}
It is known since the works of Zariski that the main case of local uniformization is that of valuations of height one. Deducing the general case is usually easy, but details depend on the method, so this principle was never formulated rigorously. Our first main result proves that, indeed, local (log) uniformization of valuations on schemes follows from the case of height one, see Theorem~\ref{unifschemes}. Induction on height is a standard method in geometry of valuations, but in this case it is not so simple and some machinery is needed. In particular, we heavily exploit various properties of blowings up.

Our second main result is Theorem~\ref{mainth} which reduces local uniformization of adic points on analytic spaces to the resolution of two problems: (a) local uniformization of analytic points, (b) local log uniformization over the residue field. In fact, any adic valuation is composed from an analytic valuation and a valuation on its residue field, and our proof combines local uniformization of both to get a local uniformization of the adic point. The argument itself is similar to the proof of Theorem~\ref{unifschemes} but more technical.

\subsubsection{Overview of the paper}
Section 2 is devoted to local log uniformization of schemes. We formulate the conjecture and reduce it to the case of height one. In Section 3 we study local uniformization on $k$-analytic spaces, which is formulated in \S3.1. Our argument uses certain log smooth but not semistable formal schemes, so we decided to use this occasion to study few basic properties of log smoothness for admissible formal $\kcirc$-schemes and formulate the general log smooth modification conjecture. The main result of the paper is then proved in \S3.4.

\subsubsection{The method}
Finally, we outline the proof of Theorem~\ref{unifschemes} in the case without divisors. This is enough to illustrate the main ideas. The starting observation is that a valuation of height at least two is composed from valuations of smaller heights which are uniformizable by induction. This allows us to easily reduce the claim to the case when a valuation $\lam$ on $X$ is composed from a valuation $\lam_0$ on $X$ with center $\eta$ and a valuation $\lam_1$ on the Zariski closure $Y=\oeta$ with center $x$, and we have that $X$ is regular at $\eta$ and $Y$ is regular at $x$. By no means this suffices to achieve our goal of desingularizing $X$ at $x$, but we can keep all these data and refine $X$ as follows. Let $E=V(b)$ be a divisor such that $X$ is smooth along $Y\setminus E$ and let $X_l$ be the blowing up of $X$ along $\cJ_l=(\cI_Y,b^l)$. One can show that $X_l\to X$ do not change $Y$ much -- the sequence of strict transforms $Y_l$ stabilizes, and if we make a simple additional assumption as in Lemma~\ref{keylem}(i), we even have that $Y_l=Y$. The main idea is that when $l$ tends to infinity, the completion of $X_l$ along $Y_l\cap E_l$ can be viewed as a tubular neighborhood of the completion of $Y_l$ along $Y_l\cap E_l$. In particular, it should only depend on $Y$ and $Y\cap E$ and instead of $X$ one could start with any regular scheme $X'$ containing $Y$ whose dimension equals the dimension of $X$. In this model case it is easy to see that each $X_l$ is regular. This expectation is worked out in key lemma~\ref{keylem} by a relatively straightforward computation of blowings up and charts. Note also that a similar idea of computing blowing up along $\cI+\cJ^l$ with a large enough $l$ was used in \cite[\S4]{temdes} to solve another desingularization problem.

The argument in the formal case is similar, and the main technical difference is that a model $\gtX_l$ obtained by blowing up a semistable $\gtX$ can be only log smooth rather than semistable, see Lemma~\ref{toriclem}.

\subsection{Conventiones}\label{conv}
Given an ideal $\cI$ on a scheme $X$ we will often use the notation $V_X(\cI)=\Spec_X(\cO_X/\cI)$ to denote its vanishing locus. If $\cI$ is generated by $t_1\.t_n\in\Gamma(\cO_X)$ then we will also use the notation $V_X(t_1\.t_n)$. The scheme $X$ can sometimes be omitted if no confusion is possible.

Blowing up of a scheme $X$ along a closed subscheme $V=V(\cI)$ will be denoted $\Bl_V(X)=\Bl_\cI(X)=\Proj_X(\oplus_{d=0}^\infty\cI^n)$. Recall that $\Bl_V(X)\to X$ is the universal morphism such that $\calI$ pullbacks to an invertible ideal. If $X=\Spec(A)$ is affine and $I\subseteq A$ is an ideal, then any $f\in I$ defines an open subscheme $X'_f=\Spec(A[\frac{I}{f}])$ of $X'=\Bl_I(X)$, where $A[\frac{I}{f}]$ denotes the $A$-subalgebra of $A_f$ generated by the elements $\frac{g}{f}$ with $g\in I$. We call $X'_f$ the {\em $f$-chart of the blowing up}. Note that $X'_f\to X$ is the universal morphism such that $I$ pulls back to the principal ideal generated by $f$, and charts $X'_{f_1}\.X'_{f_n}$ cover $X'$ if and only if $I=(f_1\.f_n)$.

\section{Local uniformization for schemes}

\subsection{Local uniformization conjectures}

\subsubsection{Valuations on integral schemes}
Let $X$ be a scheme. By a {\em semivaluation on $X$} we mean a valuation ring $R$ and a morphism $\lam\:S=\Spec(R)\to X$, and its {\em center} is the image of the closed point of $S$. Furthermore, $\lam$ is a {\em valuation} if it is dominant, in particular, $X$ is integral and one obtains an embedding $k(X)\into\Frac(R)$. In fact, we will only consider valuations in this paper. Any valuation factors uniquely through a valuation $\lam_0\:S_0\to X$ inducing an isomorphism of generic points -- just take $S_0=\Spec(R_0)$ with $R_0=k(X)\cap R$. In all our arguments below one can safely replace $\lam$ by $\lam_0$, which is often called a valuation of $k(X)$ centered on $X$ (it is determined by the valuation ring $R_0$ of $k(X)$ and the center $x=\lam(s)$, and if $X$ is separated -- only by $R_0$). We will freely use that valuations uniquely lift to any proper and birational $X$-scheme $X'$ by the valuative criterion of properness.

\subsubsection{Noetherian case}
In the sequel we will run induction on height of valuations, so the following simple result will be useful.

\begin{lem}\label{finheight}
If $\lam\:\Spec(R)\to X$ is a valuation on a noetherian scheme $X$ such that $\Frac(R)=k(X)$, then the valuation ring $R$ is of finite height $h$. In fact, $h\le d$, where $d=\dim(\calO_{X,x})$ and $x$ is the center of $\lam$.
\end{lem}
\begin{proof}
Let $\{X_\alp\}$ be the family of all modifications of $X$ and let $x_\alp$ be the center of the lift of $\lam$ to $X_\alp$. Then $R=\colim_\alp\calO_{X_\alp,x_\alp}$ and hence the approximation theory from \cite[Ch. IV, \S8.8]{ega} applies to $S=\Spec(R)=\lim_\alp\Spec(\calO_{X_\alp,x_\alp})$. In particular, the underlying topological space $|S|$ is the limit of the noetherian topological spaces whose dimension is bounded by $d$. Since $|S|$ is a chain of length $h+1$ with respect to specialization, it immediately follows that its length is bounded by the maximal length of specializing chains in $\Spec(\calO_{X_\alp,x_\alp})$, which is $d+1$.
\end{proof}

\subsubsection{Regular pairs}
Let $X$ be a scheme and $D\subseteq X$ a closed subscheme. We say that $(X,D)$ is a {\em regular pair at} a point $x\in X$ if $X$ is regular at $x$ and $D$ is an snc divisor at $x$. This terminology is not standard, especially because we exclude non-strict normal crossings, but this will be convenient. Also, if $D$ is a closed subset, then we automatically view it also as a reduced closed subscheme.

\subsubsection{Resolution along a valuation}
Since Zariski it is well understood that resolution along a valuation is the local part of the global desingularization problem. Classically, the former is called local uniformization of valuations, but we prefer to change the terminology slightly.

\begin{defin}
Let $X$ be an integral scheme and $\lam\:S\to X$ a valuation on it.

(i) By a {\em desingularization of $X$ along $\lam$} we mean a blowing up $X'=\Bl_V(X)\to X$ such that $X'$ is regular at the center $x'$ of the lift $\lam'\:S\to X'$ of $\lam$.

(ii) Let $D\subsetneq X$ be a closed subset. By a {\em log desingularization of $(X,D)$ along $\lam$} we mean a blowing up $f\:X'=\Bl_V(X)\to X$ such that the pair $(X',D'=f^{-1}(D\cup V))$ is regular at the center $x'$ of the lift $\lam'\:S\to X'$ of $\lam$.
\end{defin}

\begin{rem}
In global resolution, it suffices for many applications to only achieve that the pair $(X',f^{-1}(D))$ is regular. However, there are enough applications in which one needs the stronger version, when the exceptional divisor is added to the new boundary (or log structure) and the whole $f^{-1}(D\cup V)$ is snc. Also, such a control on the exceptional divisor is important for induction hypothesis in constructions of some desingularization methods. In this paper, as well, using the strong version of log desingularization along a valuation will turn out to be critical for local uniformization of non-archimedean spaces.
\end{rem}

\subsubsection{Local uniformizability}
Now we can introduce the central notion of this paper:

\begin{defin}\label{locunifdef}
Let $X$ be an integral scheme and $\lam\:S\to X$ a valuation.
\begin{itemize}
\item[($1$)] The valuation $\lam$ is {\em uniformizable} if there exists a cofinal family of blowings up $f_i\:X_i\to X$ such that each $f_i$ is a desingularization of $X$ along $\lam$.

\item[($2$)] The valuation $\lam$ is {\em log uniformizable} if for any closed set $D\subsetneq X$ there exists a cofinal family of blowings up $f_i\:X_i=\Bl_{V_i}(X)\to X$ such that each $f_i$ is a log desingularization of $(X,D)$ along $\lam$.
    \end{itemize}
\end{defin}

Here is the general local uniformization conjecture.

\begin{conj}\label{logunifconj}
Any valuation on a quasi-excellent integral scheme $X$ is log uniformizable.
\end{conj}

\begin{rem}\label{unifrem}
(i) Usually, one considers modifications or projective modifications in the definition of local uniformization. However, the class of blowings up is more convenient to work with in birational geometry and it is almost as general as the class of projective modifications. So, we prefer to work with it and the reader will see that this is quit beneficial.

(ii) If one establishes (log) uniformizability of valuations on a class $\bfS$ of schemes which is closed under modifications, then it suffices to prove that any $X\in\bfS$ possesses a single (log) desingularization along any valuation $\lam\:S\to X$.

(iii) Usually, it is the uniformizability property which is really useful. Originally, by local uniformization of a valuation Zariski meant only existence of a single desingularization along the valuation, but in \cite[Theorem~${\rm U}_1$]{Zariski} he proved that valuations on varieties of characteristic zero are unformizable in our sense. Also, in the introduction Zariski provided a somewhat analogous but more concrete motivation for proving this slightly stronger statement.

(iv) Log uniformizability of valuations is known for qe schemes of characteristic zero. For example, it can be deduced from the global resolution (though this is not the only way to prove it). In addition, log uniformizability was established in \cite{CP} by Cossart and Piltant for all qe schemes of dimension bounded by 3 (and then they used it to prove strong global resolution). The standard resolution conjectures imply that it should hold for the class of all qe schemes, but already uniformizability of valuations on fourfolds over a perfect fields is wide open.
\end{rem}

\subsection{Induction on height}

\subsubsection{Reduced irreducible components}
We start with a lemma which studies the situation when an integral closed subscheme $\Spec(C/J)$ of an affine scheme $\Spec(C)$ is generically an open subscheme.

\begin{lem}\label{annlem}
Assume that $C$ is a ring, $J$ is a finitely generated prime ideal with quotient $B=C/J$ and $0\neq b\in B$ is an element such that the induced morphism $\Spec(B_b)\into\Spec(C)$ is an open immersion. Then there exists $n>0$ and a lift $c\in C$ of $b^n$ such that $J=\Ann(c)$.
\end{lem}
\begin{proof}
To start with choose an arbitrary lift $s\in C$ of $b$. Since $\Spec(B_b)$ is open and closed in $\Spec(C_s)$, there exists a splitting $C_s=B_b\times B'$. Choose a presentation $(1,0)=t/s^m\in C_s$ with $t\in C$. Since $B$ is integral and $b\neq 0$, we have that $B\subseteq B_b$ and hence $t$ is mapped to $b^m$ in $B$. Thus, $ts$ is a lift of $b^{m+1}$ and $C_{ts}=(B_b\times B')_{(1,0)}=B_b$.

The localization $C\to C_{ts}$ factorizes as $C\to B\into B_b$, so has kernel $J$. Hence $J$ consists precisely of those elements that are killed by a power of $ts$. Using that $J$ is finitely generated we can find a single element $(ts)^l$ which annihilates $J$. So, $c=(ts)^l$ and $n=l(m+1)$ are as required.
\end{proof}

\subsubsection{Strict transform}
Recall that the strict transform of a closed subscheme $Z\into X$ under the blowing up $X'=\Bl_V(X)\to X$ is the schematic closure of $Z\setminus V$ in $X'$ and it easily follows from the universal property of blowings up that $Z'=\Bl_W(Z)$, where $W=V\times_XZ$. In the affine case this is compatible with charts: if $X=\Spec(A)$, $Z=\Spec(\tilA)$, $a\in A$ vanishes on $V$ and $\tila\in\tilA$ is the image of $a$, then $Z'_\tila=X'_a\times_{X'}Z'$ (recall that by our conventions $X'_a$ denotes the chart where $a$ generates $\calI_V\calO_{X'}$, and similarly for $Z'_\tila$). We will also need the following computation with strict transforms.

\begin{lem}\label{strictlem}
Assume that $X=\Spec(A)$ is an affine scheme with a closed subscheme $Z=\Spec(C)$, where $C=A/I$. Assume that $a\in A$ is an element and $X'\to X$ is the blowing up along the ideal $=I+(a)$, and let $X'_a=\Spec(A')$ be the $a$-chart of $X'$. Then,

(0) $A'=A[\frac{I}{a}]\subseteq A_a$ is the $A$-subalgebra of $A_a$ generated by the elements $t/a$ with $t\in I$.

(i) The strict transform of $Z$ is contained in the $a$-chart $X'_a$ and is isomorphic to the closed subscheme $Z'=\Spec(C/J)$, where $J=\cup_n\Ann(c^n)$ and $c$ is the image of $a$ in $C$.

(ii) If $a'\in A$ is another element whose image in $C$ equals $c$, then $a'=au$ for an element $u\in A'$ which is invertible in a neighborhood of $Z'$ in $X'$.

(iii) If $J=\Ann(c)$ in (i), then $Z'$ is the vanishing locus of the ideal $I'=a^{-1}I$ on $X'_a$, namely, $Z'=\Spec(A'/I')$.
\end{lem}
\begin{proof}
(0) This is the classical chart description of blowings up.

(i) The strict transform $Z'\to Z$ is the blowing up along the ideal $IC+aC=(c)$. Therefore $Z'=\Spec(C')$, where $C'$ is the image of $C$ in $C_c$. Clearly, $C'=C/J$. In addition, $X'$ is covered by $X'_a$ and the charts $X'_t$ with $t\in I$, and $Z'$ is disjoint from each $X'_t$ because $X'_t\times_{X'}Z'$ is the chart of the blowing up $Z'\to Z$ corresponding to the image $\tilt\in C$ of $t$, but $\tilt=0$.

(ii) We have that $a'=a+t$ with $t\in I$. Therefore, in $A'$ we have that $a'=au$, where $u=1+a^{-1}t$. We claim that it is invertible in a neighborhood of $Z'=Z'_c$ because $a^{-1}t$ vanishes on $Z'$. Indeed, $a^{-1}t$ vanishes on $Z_a=Z\times_X\Spec(A_a)$ in the localization $X_a=\Spec(A_a)$, but $Z'_c$ is by definition the schematic closure of $Z_a$ in $X'_a$, hence $a^{-1}t$ vanishes on the whole $Z'_c$.

(iii) The homomorphism $\phi\:A'\to C'$ is onto and we should prove that $\Ker(\phi)=I'$. We showed in the proof of (ii) that for each $t\in I$ the element $t'=a^{-1}t\in I'$ vanishes on $Z'_c$, and hence $I'\subseteq\Ker(\phi)$.

Conversely, $A'=A+I'$, hence it suffices to show that any $x\in A\cap\Ker(\phi)$ lies in $I'$. Since $x$ is mapped to $0$ in $C'$, its image in $C$ is contained in $J$ and hence is annihilated by $c$. This implies that $ax\in I$ and we obtain that $x\in I'$.
\end{proof}

\subsubsection{The key lemma}
When running induction on height of valuations in the local uniformization problem one naturally arrives at the situation, when a scheme $X$ is regular at a point $\eta$ and the Zariski closure $Y$ of $\eta$ is regular at a point $x$. The basic obstacle for an induction step is that $X$ does not have to be regular at $x$ (the simplest example is obtained by taking $x$ the usual cone singularity and $Y$ a line passing through the origin). This forces us to look for a modification which partially resolves $x$, and, fortunately, it suffices to do this along the strict transform of $Y$. Slightly surprisingly, this turns out to be a simple task, at least, under a mild technical assumption.

\begin{lem}\label{keylem}
(i) Let $X$ be a noetherian scheme, $Y\into X$ an integral closed subscheme with generic point $\eta$ and $x\in Y$ a point. Assume that $X$ and $Y$ are regular at $\eta$ and $x$, respectively, and there exist elements $\ut=(t_1\.t_n)$ in $\calO_{X,x}$ whose images form a regular family of parameters of $\calO_{X,\eta}$. In particular, $D=V(t_1\dots t_n)$ is snc at $\eta$. Then there exists a blowing up $f\:X'=\Bl_V(X)\to X$ such that the strict transform $g\:Y'\to Y$ is an isomorphism over $x$ and $X'$ is regular at $x'=g^{-1}(x)$.

(ii) Assume, in addition, that there exists a divisor $E\subset Y$ which is snc at $x$ and such that the pair $(X,D)$ is regular at any point of $Y\setminus E$ which generizes $x$. Then in addition to the assertion of (i) one can achieve that the closed set $D'=f^{-1}(D\cup V)$ provided with the reduced scheme structure is an snc divisor at $x'$.
\end{lem}
\begin{proof}
First, we observe that it suffices to establish the local case when $X$ coincides with the localization $X_x=\Spec(\calO_{X,x})$. Indeed, all assumptions of the lemma and conditions the blowing up $f$ should satisfy are local at $x$, and once an appropriate blowing up $\Bl_W(X_x)\to X_x$ is constructed we can extend it to a blowing up of the whole $X$ just by blowing up the schematic closure $V$ of $W$ in $X$ (we use the simple fact that $V\times_XX_x=W$ and blowings up are compatible with flat morphisms).

Thus, we assume in the sequel that $X$ is a local scheme: $X=\Spec(A)$ where $A=\cO_{X,x}$. In particular, $Y$ is a regular local scheme and in the case of (ii) $E$ is an snc divisor on $Y$. Let $I$ be the ideal defining $Y$, so $Y=\Spec(B)$ with $B=A/I$, and in case (ii) let $b\in B$ be an element such that $E=V_Y(b)$. We will argue in both cases simultaneously, so when dealing with (i) just fix a large enough divisor $E=V_Y(b)$ on $Y$, not necessarily snc, such that the pair $(X,D)$ is regular at any point of $Y\setminus E$. Now, the idea is very simple: take $a\in A$ to be a lift of $b^\ell$ with a large enough $\ell$ and blow up $X$ along $V(\ut,a)$. Let us work this plan out.

Set $C=A/(\ut)$ and $Z=\Spec(C)$. In particular, $Y$ is a closed subscheme of $Z$ and $B=C/J$ for the prime ideal $J=IC$. Note that $Y_b=\Spec(B_b)$ is open in $Z$. Indeed, $Z$ contains an irreducible component with generic point $\eta$, hence $Y\into Z$ is a closed immersion of codimension 0. Moreover, $D$ is snc and $Z$ is its stratum of maximal multiplicity at any point $z\in Y_b$. Therefore $Z$ is integral at $z$ and necessarily $Y\into Z$ is an open immersion at $z$. By Lemma~\ref{annlem} there exists $\ell>0$ and a lift $c\in C$ of $b^\ell$ such that $J=\Ann(c)$. Choose any lift $a\in A$ of $c$ and consider the blowing up $f\:X'\to X$ along $V=V(\ut,a)$.

Let $h\:Z'\to Z$ be the strict transform with respect to $f$. By Lemma \ref{strictlem}(i), $Z'$ is contained in the $a$-chart $X'_a=\Spec(A')$, where $A'=A[\ut']$ with $t'_i=t_i/a$, and, $Z'=\Spec(C/J)=Y$. In particular, $x'=h^{-1}(x)$ is a single point, and $Y'$ is a closed subscheme containing the generic point $\eta$ of the integral scheme $Z'$ and hence $Y'=Z'=Y$. In addition, note that $J=\Ann(c^m)$ for any $m>0$ because $J=\Ann(c)$ is prime, and hence $Z'$ is the vanishing locus of $\ut'$ by part (iii) of the same lemma. Thus, $A'/(\ut')=B$.

Let $\us=(s_1\.s_m)$ be a set of elements of $A'$ whose image $\os_1\.\os_m\in B$ is a family of regular parameters of the regular local ring $B$. Then $(\ut',\us)$ is a regular family of parameters of $\calO_{X',x'}$ because this set generates $m_{x'}$ in the obvious way and $\dim(\calO_{X',x'})\ge\dim(\calO_{X',\eta})+\dim(\calO_{Y',x'})=n+m$. This proves that $X'$ is regular at $x'$, finishing the proof of (i).

In case (ii) we choose the parameters $s_i\in A'$ more specifically, namely we choose them so that $E=V(\os_1\dots\os_r)$ for $r\le m$ and $b=\prod_{i=1}^r\os_i$. Establishing (ii) we should also care for the preimage of $D\cup V$. We are only interested in studying the chart $f_a\:X'_a\to X$, so set $D'=f_a^{-1}(D\cup V)$. First, let us see what can go wrong. Since $V\subset D$ we have that on the level of sets $$D'=f_a^{-1}(D)=V(t_1\dots t_n)=V(a^nt'_1\dots t'_n)=V(t'_1\dots t'_n)\cup V(a).$$ The component $V(t'_1\dots t'_n)$ (which is, in fact, the strict transform of $D$) is snc at $x'$ because $\ut'$ is a partial family of regular parameters, but concerning the exceptional component $E'=V(a)$ we only know that its restriction onto $Z'=Y$ is given by the vanishing of the image $b^\ell\in B$ of $a$, namely, $E'\times_{X'}Y=\Spec(B/(b^\ell))$. The divisor $D'$ (and the component $E'$) can be very singular, so we should improve $X'$ further.

Next, we outline the idea without proofs (though the interested reader can easily check our assertion). It turns out that it suffices just to replace $X'$ by its blow up $X''$ along the intersection $E'\times_{X'}Y$. In fact, one blows up the center which is given at $x'$ by $(\ut',a)$ obtaining the new coordinates $t''_i=a^{-1}t'_i=a^{-2}t_i$. This also suggests the argument which we use below -- instead of the sequence of two blowings up we will just blow up $(\ut,a^2)$ at once. So, consider the blowing up $g\:X''\to X$ along $(\ut,a^2)$ and let us prove that it satisfies all assertions of the lemma.

Consider the $a^2$-chart $g_{a^2}\:X''_{a^2}=\Spec(A'')\to X$, where $A''=A[\ut'']$ with $t''_i=t_i/a^2$.  The same argument as was used to study $f$ shows that the strict transform $Z''$ of $Z$ is isomorphic to $Y$, it is given by the vanishing of $\ut''$ and lies in $X''_{a^2}$, and one has that $$D''=g^{-1}_{a^2}(D\cup V)=g^{-1}_{a^2}(D)=V(t''_1\dots t''_n)\cup V(a).$$ We claim that $D''$ is snc at the preimage $x''\in Z''$ of $x$. Recall that $a$ is a lift of $b^\ell\in B$ to $A$, hence the same is true for its image in $A'$, which we denote by the same letter $a$. As we remarked above $V_{X'_a}(a)$ can be very singular. On the other hand, $a'=\prod_{i=1}^rs_i^\ell$ is another lift of $b^\ell=\prod_{i=1}^r\os_i^\ell$ to $A'$ and $V_{X'_a}(a')$ is snc at $x'$. The point is that after the blowing up $X''\to X'$ the strict transforms of $V_{X'_a}(a)$ and $V_{X'_a}(a')$ coincide locally at $x''$ by Lemma~\ref{strictlem}(ii). Indeed, $A''=A'[\ut'/a]$, hence $X''_a$ is also the $a$-chart of the blowing up of $X'$ along $(\ut',a)$, and by Lemma~\ref{strictlem}(ii) we obtain that $a=va'$ for a unit $v\in\calO_{X'',x''}$. In particular, in a neighborhood of $x''$ we have that $V_{X''_a}(a)=V_{X''_a}(a')=\cup_{i=1}^r V_{X''_A}(s_i)$. The tuple $(\ut'',\us)$ is a regular family of parameters at $\calO_{X'',x''}$ (by the same argument as was used for $(\ut',\us)$ and $\cO_{X',x'}$), hence $D''$ is indeed an snc divisor at $x''$.
\end{proof}

\begin{rem}
It looks tempting to simplify the above proof by carefully choosing $a$ from the beginning. Namely, the homomorphism $A\to B$ is surjective, hence instead of the lifts $s_i\in A'$ of $\os_i$ we can even take $s_i\in A$. At first glance, it seems that in such a case simply taking $a=\prod_{i=1}^rs_i^l$ one guarantees that it is already split into a product, and hence a single blowing up of $(\ut,a)$ suffices. However, the problem is that our choice of the lift $a\in A$ went through a choice of the lift $c\in C$ of $b=\prod_{i=1}^r\os_i$ such that $J=\Ann(c)$. We used Lemma~\ref{annlem} to find it, and this imposes some restrictions. We do not know if a tricky choice of $a$ can remedy this problem, but it is not important for our goals.
\end{rem}

\subsection{The main theorem for schemes}
Here is the main result of this paper for local uniformization of schemes. In particular, it reduces Conjecture~\ref{locunifconj} to the case of valuations of height one, but we prefer a more general formulation, which can be applied to various classes of schemes. For example, it applies to the class of algebraic varieties, or varieties of dimension bounded by some number.jdjhdj

\begin{theor}\label{unifschemes}
Let $\bfS$ be a class of qe schemes closed under modifications and closed immersions. If any valuation of height one on an integral scheme from $\bfS$ is uniformizable or log uniformizable, then the same is true for any valuation on integral schemes from $\bfS$.
\end{theor}
\begin{proof}
For clarity we will consider the logarithmic case, as the argument in the non-logarithmic case is obtained by omitting some parts of the construction. We will indicate the required changes at the end of the proof.

By Remark~\ref{unifrem}(ii) it suffices to prove that if $X\in\bfS$ is an integral scheme, $D\subsetneq X$ a closed subset and $\lam\:S=\Spec(R)\to X$ a valuation, then there exists a desingularization $f\:X'=\Bl_V(X)\to X$ of $(X,D)$ along $\lam$. Replacing $R$ by $R\cap k(X)$ we can assume that $k(S)=k(R)$. If $R$ is of height 1, then $f$ exists by our assumption, so assume that it is of height $h\ge 2$. Recall that $h$ is finite by Lemma~\ref{finheight}. Choose a non-zero and non-maximal prime ideal $p\subset R$ and consider the valuation rings $R_0=R_p$ and $\oR=R/p$. Then $S$ is covered by (in fact, pushed out from) valuation subschemes of smaller height $S_0=\Spec(R_0)$ and $\oS=\Spec(\oR)$, which are open and closed, respectively. In particular, valuations of $S_0$ and $\oS$ on schemes from $\bfS$ are log uniformizable by the induction assumption. Let $\eta\in X$ be the center of the induced valuation $\lam_0\:S_0\to X$ and let $Y$ be the Zariski closure of $\eta$, then $\lam$ induces a valuation $\olam\:\oS\to Y$ (here $k(Y)\subseteq k(\oS)$ does not have to be an equality).

A desingularization along $\lam$ will be constructed by composing a tower of blowings up which gradually improve the properties of $X$ and $D$ at the center of $\lam$. To simplify notation, we will replace $X$ by the modification constructed so far. So at each step we will construct a blowing up $f\:X'=\Bl_W(X)\to X$, and then simplify notation by replacing $X$ by $X'$ and $D$ by $D'=f^{-1}(D\cup W)$. This makes sense because composition of blowings up is a blowing up. By $\eta$ and $x$ we will always denote the centers of the lifts of $\lam_0$ and $\lam$, respectively, to the current $X$. Finally, by $Y$ and $\olam\:\oS\to Y$ we denote the Zariski closure of $\eta$ and the lift of $\olam$ to $Y$ (of course, $\olam$ is the restriction of $\lam$).

Step 1. {\it Replacing $X$ by its blowing up we can assume that $(X,D)$ is regular at $\eta$.} Indeed, just by the induction assumption applied to $\lam_0$, there exists a blowing up $f\:X'=\Bl_V(X)\to X$ such that $(X',D')$ is regular at $\eta$.

Step 2. {\it In addition to the condition of step 1, one can achieve that there exist elements $t_1\.t_n\in\calO_{X,x}$ whose images form a regular family of parameters of $\calO_{X,\eta}$ and such that $D=V(t_1\dots t_n)$ locally at $\eta$.} Since $D$ is snc at $\eta$, locally at $\eta$ it is of the form $V(t_1\dots t_m)$ where, $t_1\. t_m\in\calO_{X,\eta}$ form a partial family of regular parameters. Complete this family to a regular family of parameters $t_1\.t_n\in\calO_{X,\eta}$ and let $D_i$ be the schematic closure of $\Spec(\calO_{X,\eta}/(t_i))$ in $X$. Loosely speaking we will simply blow up the Weil divisors $D_i$ making them Cartier and hence achieve that up to a unit each $t_i$ extends to $\calO_{X,x}$. This will not modify $X$ at $\eta$, but will increase $D$.

Consider the ideals $\calI_i=\calI_{D_i}\subset\calO_{X}$ corresponding to $D_i$, their product $\calI=\prod_{i=1}^n\calI_i$ and the corresponding closed subscheme $W=V_{X}(\calI)$. The blowing up $g\:X'=\Bl_W(X)\to X$ makes the pullback of each $\calI_i$ an invertible ideal, that is, each $D'_i=D_i\times_{X}X'$ is a Cartier divisor. Therefore we can choose $t_i\in\calO_{X',x'}$ such that $D'_i=V(t_i)$ locally at $x'$.

Each $\calI_i$ is principal at $\eta$, hence $X'\to X$ is an isomorphism over $\eta$. Locally at $\eta$ we have that $D$ is a subset of  $W=\cup_{i=1}^nD_i$, hence locally at $\eta'$ the set $D'$ is the preimage of $W$. Using that $W=V(t_1\dots t_n)$ locally at $\eta$, we obtain that $D'=V(t_1\dots t_n)$ locally at $\eta'$, completing the step.

Step 3. {\it Let $D_i$ be the schematic closure of $\Spec(\calO_{X,x}/t_i\cO_{X,x})$. Then there exists a divisor $E\subset Y$ such that $(X,\cup_{i=1}^nD_i)$ is a regular pair at any point of $Y\setminus E$.} This is clear since $\cup_{i=1}^nD_i$ is snc at $\eta$.

Step 4. {\it In addition to the conditions of steps 1--3 we can achieve that $(Y,E)$ is a regular pair at $x$.} Applying the induction assumption we can find a blowing up $h\:Y'=\Bl_W(Y)\to Y$, which desingularizes $(Y,E)$ along $\olam$. Thus, $(Y',E'=h^{-1}(E\cup W))$ is regular at the center $x'$ of the lift $\olam'\:\oS\to Y'$ of $\olam$. Using the same center $W$ we obtain a blowing up $g\:X'=\Bl_W(X)\to X$ such that $Y'\into X'$ is the strict transform of $Y$ and the lift $\lam'\:S\to X'$ of $\lam$ restricts to the valuation $\olam'$ on $Y'$. In particular, $x'$ is the center of $\lam'$. Thus, the new condition of the step is satisfied, but we also have to check that the conditions of steps 1-3 were not destroyed.

Let $D'_i$ denote the Zariski closure of $\Spec(\calO_{X',x'}/t_i\calO_{X',x'})$. For any point $z'\in Y'\setminus E'$ we have that $X'\to X$ is a local isomorphism at $z'$ and takes it to $z\in Y\setminus E$. Since $(X,\cup_{i=1}^nD_i)$ is a regular pair at $z$ we obtain that $(X',\cup_{i=1}^nD'_i)$ is a regular pair at $z'$. Thus, $x'$ satisfies conditions of step 3. Conditions of steps 1 and 2 are satisfied because $\eta\notin W$, and hence locally at $\eta'$ and $\eta$ the pairs $(X',\cup_{i=1}^nD'_i)$ and $(X,\cup_{i=1}^nD_i)$ are isomorphic.

Step 5. {\it End of proof: one can achieve that $X'\to X$ provides a desingularization of $(X,D)$ along $\lam$.} To complete the argument we apply Lemma~\ref{keylem}(ii) to $x\in Y\subset X$, $E$ and $t_1\.t_n\in\calO_{X,x}$, obtaining a blowing up $f\:X'=\Bl_W(X)\to X$ such that the strict transform $Y'$ of $Y$ is isomorphic to $Y$, the preimage $x'\in Y'$ of $x$ is a regular point of $X'$, and $D'=f^{-1}(D\cup W)$ is snc at $x'$. Clearly, $x'$ is the center of the lift $\lam'\:S\to X'$ of $\lam$, hence $X'\to X$ is as required.

Finally we discuss the simplifications in the non-logarithmic case. They are as follows: one does not consider $D$ and $E$, one only worries for regularity of schemes and not pairs, and step 3 is not needed. For example, step 4 achieves that $Y$ is regular at $x$. However, one does introduce $t_1\.t_n\in\calO_{X,x}$ in step 2, and in step 5 one uses only part (i) of Lemma~\ref{keylem} with $x\in Y\subset X$ and $t_1\.t_n$ being the inputs.
\end{proof}

\section{Local uniformization for non-archimedean spaces}
In this section we fix a complete non-trivially valued real-valued field $k$ and denote by $\kcirc$, $\kcirccirc$ and $\tilk=\kcirc/\kcirccirc$ the ring of integers, the maximal ideal of $\kcirc$ and the residue field. By $\pi$ or $\varpi$ we denote any {\em quasi-unifrormizer} of $k$, that is, an element $\pi\in\kcirccirc\setminus\{0\}$. By {\em $k$-analytic spaces} we always mean Berkovich $k$-analytic spaces introduced in \cite[Chpater~1]{berihes}. We will only consider strictly $k$-analytic spaces, so the word ``strictly'' will usually be omitted. For any extension of complete real valued fields $l/k$ we set $X_l=X\wtimes_kl$.

\subsection{The local uniformization conjecture}\label{locunifsec}
We start with recalling some terminology and introducing the local uniformization conjecture.

\subsubsection{G-topology}
Each strictly $k$-analytic space is provided with the $G$-topology of strictly analytic domains. It is finer than the usual topology and it is used to define coherent sheaves on $X$, especially, when $X$ is not good. The site $X_G$ is equivalent to a topological space $|X_G|$ of all its points, and this space possesses a very natural interpretation. First, $|X_G|=X^\ad$, where $X^\ad$ is the Huber's adic space corresponding to $X$. Furthermore, $X\subseteq X_\ad$ and $X$ is the maximal locally Hausdorff quotient of $X^\ad$ with $r\:X^\ad\to X$ being the maximal generization map. Finally, the fiber of $X^\ad\to X$ over $x\in X$ is the germ reduction $\tilX_x$ from \cite{local-properties}. All in all, points of $X^\ad$ are interpreted as semivaluations and $y\in r^{-1}(x)$ can be interpreted as a valuation on $\wHx$ which is trivial on $\tilk$ and is centered on formal models of $X$.

\subsubsection{Formal models}
By an {\em admissible formal $\kcirc$-scheme} we mean a flat topologically finitely presented formal $\kcirc$-scheme and by $\gtX_s=\gtX\otimes_{\kcirc}\tilk$ we denote the (not necessarily reduced) {\em closed fiber} of $\gtX$. There is a natural generic fiber functor which associates to each admissible formal $\kcirc$-scheme a compact $k$-analytic space $X_\eta$. Locally it is defined by $\Spf(A)_\eta=\calM(A\otimes_{\kcirc}k)$ and the global construction is obtained by patching. A {\em formal model} of an analytic space $X$ is an admissible formal $\kcirc$-scheme $\gtX$ provided with an isomorphism $X=\gtX_\eta$. It admits natural specialization maps $X\into X^\ad\to\gtX$ from the analytic and adic generic fibers. For any complete extension $l/k$ we have that $\gtX_{\lcirc}:=\gtX\wtimes_{\kcirc}\lcirc$ is a formal model of $X_l$. An {\em admissible blowing up} of $\gtX$ is a formal blowing up along an open ideal $\calI\subseteq\calO_\gtX$, it does not modify the generic fiber (equivalently, it induces a blowing up of the trivial ideal $\calI\calO_X=\calO_X$).

\begin{rem}
A famous theorem of Raynaud states that the generic fiber functor induces an equivalence between the category of admissible formal schemes localized by the class of all admissible blowings up and the category of compact strictly $k$-analytic spaces. Loosely speaking, it asserts that any compact strictly $k$-analytic space possesses a formal model which is unique up to an admissible blowing up, and every morphism between compact strictly analytic spaces arises from a morphism between suitable formal models. In particular, $X^\ad$ is nothing else but the projective limit of all formal models of $X$.
\end{rem}

\subsubsection{Smooth morphisms}
Recall that a morphism $\gtX\to\gtY$ of admissible formal schemes is {\em smooth} (resp. {\em \'etale}, resp. {\em flat}, resp. {\em unramified}), if for any ideal of definition $\calJ\subset\calO_\gtY$ the morphism of schemes $(\gtX,\calO_\gtX/\calJ)\to(\gtY,\calO_\gtY/\calJ\calO_\gtY)$ is smooth (resp. \'etale, resp. flat, resp. unramified). In fact, it suffices to test a single ideal of definition and, as we show below, one can even work with the closed fiber, which does not correspond to an ideal of definition, but has the advantage of being a morphism of $\tilk$-varieties.

\begin{lem}\label{simpleetale}
Let $\gtX $ and $\gtY$ be admissible formal $\kcirc$-schemes. A $\kcirc$-morphism $f:\gtX\to\gtY$ is smooth (resp. \'etale, resp. flat, resp. unramified) if and only if its special fiber $f_s:\gtX_s\to\gtY_s$ is smooth (resp. \'etale, resp. flat, resp. unramified).
\end{lem}
\begin{proof}
Since ideals of the form $\pi\calO_\gtX$ with a pseudo-uniformizer $\pi$ are cofinal among all ideals of definitions, it suffices to show that $f_\pi=f\otimes_{\kcirc} \kcirc/\pi\kcirc$ is smooth (resp. \'etale, resp. flat, resp. unramified) if and only if $f_s=f_\pi\otimes\tilk$ satisfies the same property. This follows from an appropriate fiberwise criterion, see \cite[${\rm IV}_4$ Propositions~17.8.1 and 17.8.2]{ega} and \cite[${\rm IV}_3$ Propositions~11.3.10]{ega}
\end{proof}

\subsubsection{Semistable formal schemes}
For any $d\ge 0$ and a quasi-uniformizer $\pi$ set $$\gtS_{\pi,d}=\Spf\left(\kcirc\{t_0\.t_d\}/(t_0\ldots t_d-\pi)\right).$$ These are the standard (or model) semistable formal schemes, and in general one says that an admissible formal scheme $\gtX$ is {\em semistable} at a point $x$ (or $x$ is a semi-stable point) if for some \'etale neighborhood $\gtU$ of $x$ there exists a smooth morphism $f:\gtU\to\gtS_{\pi,d}$ called a {\em semistable chart} at $x$. If one can choose $\gtU$ to be a usual neighborhood, then $\gtX$ is {\em strictly semistable} at $x$.

\begin{rem}
Usually one requires $f$ to be \'etale, but this leads to an equivalent definition because a smooth chart factors through an \'etale morphism $\gtU\to\gtS_{\pi,d}\times\bfA^n_{\kcirc}$ and it is easy to see that the target locally admits \'etale morphisms to $\gtS_{\pi,d+n}$. This observation is (more or less) a particular case of Kato's observation that the chart criterion of log smoothness can use only \'etale charts.
\end{rem}

\subsubsection{Semistable parameters}
It will be convenient to work with charts that take a given point to the origin $O=V(t_0\.t_d)\in\gtS_{\pi,d}$, and this motivates the following terminology. Let $\gtX$ be an admissible formal $\kcirc$-scheme and $x\in\gtX$ a point. We say that $t_0\.t_m\in\cO_{\gtX,x}$ are {\em twisted semistable parameters} at $x$ if $t_0\dots t_m=u\pi$, where $\pi$ is a pseudo-uniformizer, $u\in\calO_{\gtX,x}^\times$ is a unit, and the subscheme $V=V_{\gtX_s}(t_0\.t_m)$ is smooth and of codimension $m$ at $x$. If $u=1$ then we say that $t_0\.t_m$ are {\em semistable parameters}. Naturally, we will usually work with semistable parameters, but the twisted notion will be occasionally used for technical reasons. Here are few simple but useful observations.

\begin{rem}\label{semirem}
(i) If $t_0\.t_m$ are semistable parameters and $u_0\.u_m\in\cO_{\gtX,x}^\times$ are such that $u_0\dots u_m=1$, then $t'_i=u_it_i$ form another family of semistable parameters.

(ii) If $t_0\.t_m$ are twisted semistable parameters, then replacing $t_0$ by $t_0/u$ one obtains a semistable family of parameters.

(iii) Assume that $\gtX\to\gtS_{\pi,d}$ is smooth at $x$ and given by $t'_0\.t'_d$ ordered in such a way that $t'_i(x)=0$ if and only if $i\le m$ for some $0\le m\le d$. Then locally at $x$ the elements $t_0=t'_0/(t'_{m+1}\dots t'_d)$ and $t_i=t'_i$ for $1\le i\le m$ form a semistable family and the induced morphism $\gtU\to\gtS_{\pi,m}$ is smooth at $x$. Indeed, we have just described a projection $\gtS_{\pi,d}\setminus V(t_{m+1}\dots t_d)\to\gtS_{\pi,m}$ whose smoothness is easily verified on the level of closed fibers.
\end{rem}

In fact, the last observation of part (iii) above is a general property which characterizes semistable parameters, but the proof is not completely obvious. We will prove it later as a particular case of Theorem~\ref{logsmth}, but prefer to formulate it now for expositional reasons. This does not cause to a circular reasoning because this lemma and its corollary will only be used in \S\ref{mainsec}.

\begin{lem}\label{semistlem}
Let $\gtX$ be a formal admissible $\kcirc$-scheme with a point $x\in\gtX$. Then $t_0\.t_m\in\calO_{\gtX,x}$ is a family of semistable parameters at $x$ if and only if locally at $x$ they induce a semistable chart $f\:\gtU\to\gtS_{\pi,m}$ which takes $x$ to the origin.
\end{lem}
\begin{proof}
The definition of semistable parameters just says that a morphism $f$ exists and the fiber $V$ over the origin is smooth at $x$ and of correct codimension. This immediately implies the inverse implication, but proving that $f$ is smooth at $x$ requires an additional argument (in particular, taking the codimension into account). This implication is a particular case of Theorem~\ref{logsmth}.
\end{proof}

\begin{cor}\label{semistabledescent}
Let $\gtX$ be an admissible formal $\kcirc$-scheme, $K=\whka$ and $f_{\Kcirc}\:\gtX'_{\Kcirc}\to\gtX_{\Kcirc}$ an admissible blowing up. Then there exists a finite separable extension $l/k$ and an admissible formal blowing up $f_{\lcirc}\:\gtX'_{\lcirc}\to\gtX_{\lcirc}$ such that $f_{\Kcirc}$ is the pullback of $f_{\lcirc}$. Furthermore, $\gtX'_{\Kcirc}$ is strictly semistable at a point $x_K$ if and only if for a large enough finite separable $l'/l$ the formal model $\gtX'_{l'^\circ}$ is strictly semistable at the image $x_{l'}$ of $x_K$.
\end{cor}
\begin{proof}
The center $\calI_K\subseteq\calO_{\gtX_{\Kcirc}}=\calO_{\gtX}\wtimes_{\kcirc}\Kcirc$ of $f_{\Kcirc}$ is open, hence it is locally generated by a family of the form $\pi,t_1\.t_n$ with $0\neq \pi\in\kcirc$, and replacing each $t_i$ by $t_i+a_i\pi$ we can assume that they are contained in $\calO_{\gtX}\otimes_{\kcirc}(k^a)^\circ$, and hence also in some $\calO_{\gtX}\otimes_{\kcirc}\lcirc=\calO_{\gtX_{\lcirc}}$ with a large enough finite separable extension $l/k$. This proves that $\calI_K$ descends to an ideal $\calI_l\subseteq\calO_{\gtX_{\lcirc}}$ and the blowing up $f_{\lcirc}\:\gtX'_{\lcirc}\to\gtX_{\lcirc}$ along $\cI_l$ pulls back to $f_{\Kcirc}$. Moreover, this is true for any larger extension of $k$.

In the second part only the direct implication needs a proof, so assume that $\gtX'_{\Kcirc}$ is strictly semistable at $x_K$ and choose a family of semistable parameters $t_0\.t_m$ at $x_K$. Each $t_i$ divides $\pi$ and it follows that for any $t'_i$ close enough to $t_i$ we have that $t'_i=u_it_i$ with $u_i\in 1+\Kcirccirc\calO_{\gtX'_{\Kcirc},x_K}$ a principal unit (in fact, it suffices that $t_i-t'_i\in\pi\Kcirccirc\cO_{\gtX'_{\Kcirc},x_K}$). In particular, we can choose $t'_1\.t'_m$ as above so that $t'_i\in\calO_{\gtX'_{l'^\circ}}=\calO_{\gtX'_{\lcirc}}\otimes_{\lcirc}l'^\circ$ for an appropriate $l'/l$. By Remark~\ref{semirem}(i), $t'_0=\pi/(t'_1\dots t'_m)$ and $t'_1\.t'_m$ form another family of semistable parameters at $x_K$, and we will show that they also form such a family at $x_{l'}$ and hence the latter is a strictly semistable point by Lemma~\ref{semistlem}.

Since $t'_1\dots t'_m$ divides $\pi$ in $\calO_{\gtX'_{\Kcirc},x_K}$ and $\gtX'_{\Kcirc}\to\gtX_{l'^\circ}$ is faithfully flat, $t'_1\dots t'_m$ also divides $\pi$ in $\calO_{\gtX'_{l'^\circ},x_{l'}}$. The ratio coincides with $t'_0$ because by the admissibility assumption $\calO_{\gtX'_{\Kcirc},x_K}$ has no $\pi$-torsion. Hence $t'_0\.t'_m$ are defined at $x_{l'}$, and the morphism they define to $\Spf(l'^\circ\{t'_0\.t'_m\}/(t'_0\dots t'_m-\pi))$ is smooth at $x_{l'}$ because its base change is smooth at $x_K$. Thus, $\gtX'_{l'^\circ}$ is strictly semistable at $x_{l'}$.
\end{proof}

\subsubsection{Uniformizability}
We say that a $k$-analytic space is {\em rig-smooth} if it is smooth at any Zariski closed (or rigid) point. Note that the usual notion of smoothness in Berkovich geometry assumes also that $X$ is boundaryless, but there is a notion of quasi-smoothness which does not make this assumption. For strictly analytic spaces it is equivalent to rig-smoothness.

\begin{defin}
Assume that $X$ is a rig-smooth analytic space. A point $y\in X^\ad$ is {\em uniformizable} if for any formal model $\gtX$ there exists a finite separable extension $l/k$ and an admissible blowing up $\gtX'\to \gtX_l$ such that a preimage of $y$ in $X_l^\ad$ specializes to a semistable point of $\gtX'$.
\end{defin}

\begin{rem}
(i) It is known since the semistable reduction of Deligne-Mumford that even for curves in the general case one has to extend the ground field and one cannot expect smooth specializations to exist, so such uniformizability is the best one can hope for.

(ii) As in the case of varieties, we chose the formulation where a cofinal family of good models exists and local resolution is achieved by a blow up. The former is absolutely critical to have a reasonable notion. The later is a convenient in applications property establishing which should not cause serious additional difficulties.
\end{rem}

\subsubsection{Local uniformization conjecture}
Now we can formulate the analytic analogue of Zariski local uniformization.

\begin{conj}\label{locunifconj}
Assume that $X$ is a rig-smooth $k$-analytic space. Then any point $x\in X^\ad$ is uniformizable.
\end{conj}

This is the version of local uniformization one usually tries to attack by studying the situation locally. Since for a finite Galois $l/k$ all preimages of a point of $X^\ad$ in $X^\ad_l$ are conjugate, the conjecture immediately implies the following less local version, which is formulated in terms of analytic spaces only.

\begin{conj}\label{locunifcor}
Assume that $X$ is a rig-smooth compact $k$-analytic space. Then there exists a finite extension $l/k$ and a finite covering $X_l=\cup_{i=1}^nX_i$ such that each $X_i$ possesses a semistable formal model $\gtX_i$ over $l$.
\end{conj}

This is the finiteness (or admissibility) of the covering condition which replaces uniformization of adic non-analytic points in the local version of the conjecture.

It is important to consider arbitrary ground fields $k$ in the applications, so the conjectures are formulated up to a ground filed extension. However, it immediately follows from Corollary~\ref{semistabledescent} that it suffices to prove them over an algebraically closed $k$, when no ground field extension is needed. In fact we even have the following slightly more precise property.

\begin{lem}\label{unifdescent}
Assume that $X$ is a rig-smooth compact $k$-analytic space, $x\in X^\ad$ a point and $y$ a preimage in $X^\ad_{\whka}$. Then $x$ is uniformizable if and only if $y$ is uniformizable.
\end{lem}

\begin{rem}\label{locunifrem}
(0) The conjecture is known in characteristic zero, because (using Elkik's theorem) it can be deduced from the global resolution provided by \cite[Theorem~1.2.19(2)]{ATW-relative}. In addition, it can be proved much easier along the lines of \cite{temaltered}, but this was never worked out in detail.

(1) The conjecture is wide open when $d=\dim(X)>1$ and $\cha(\tilk)>0$. It seems plausible that the difficulty of proving this conjecture for some $d$ is the same or a bit easier than the difficulty of proving the local uniformization conjecture of varieties in dimension $d+1$.

(2) For $d=1$ local uniformization implies semistable reduction of curves rather straightforwardly. Probably, the case of $d=2$ will lead to a proof of the global resolution of admissible formal surfaces, though the reduction will be difficult.

(3) So far, the only known result in any dimension is the weaker version, which only asserts existence of an \'etale cover $X'\to X_l$ such that $X'$ possesses a semistable model over $l$, see \cite[Theorem~3.4.1]{temaltered}. To some extent this weakening is analogous to de Jong's resolution by a separable alteration.
\end{rem}

\subsection{Log smooth formal $\kcirc$-schemes}
In the sequel, we will need to work with certain log smooth formal schemes which are not semistable. Since this does not make the arguments essentially more difficult, we decided to consider general log smooth formal schemes and use this occasion to also formulate the global counterpart of the local uniformization conjecture.

\subsubsection{On log structures on admissible formal schemes}\label{modelsec}
There are two natural approaches to provide admissible formal schemes with log structures. The first one is to work with fine or fs log structures, as in the classical algebraic setting. It is technically easier, but depends on some choices, including the choice of a sufficiently large fine log structure on $\gtS=\Spf(\kcirc)$. This is the approach we choose in this paper. In particular, our model blocks will be of the form $\gtS_P\{Q\}:=\Spf(\kcirc_P\{Q\})$ with fine $P\subseteq Q$.

\begin{rem}
For the sake of completeness we indicate that another approach consists in providing an admissible formal scheme $\gtX$ with the canonical log structure $M_\gtX$ which associates to an open $\gtU$ the monoid of all functions $f\in\Gamma(\calO_\gtU)$ such that the ideal $f\calO_\gtU$ is open. In other words, $M_\gtX$ is the monoid of functions which become units on the generic fiber $\gtX_\eta$, so informally $M_\gtX=i_*(\calO^\times_{\gtX_\eta})\cap\calO_\gtX$, where $i$ is the ``embedding'' $\gtX_\eta\into\gtX$ (this does make a formal sense in adic geometry). Moreover, one can also consider log structures which are finitely generated over the canonical ones, thus allowing non-trivial fine log structures on the generic fibers. The canonical log structure is not fine even for $\gtS$ (unless $\kcirc$ is a DVR) and the building blocks are completions of toric $\kcirc$-schemes defined in \cite{Gubler-Soto}. In fact, many questions about canonical log structures can be reduced to the fine case by approximating by fine substructures, but instead of working this out we prefer to work with fine log structures right ahead.
\end{rem}

\subsubsection{Model formal schemes}
Assume that $\lam\:P\into Q$ is an injective local homomorphism of fine monoids and $\pi\:P\into\kcirc\setminus\{0\}$ an embedding which will be written exponentially: $p\mapsto\pi^p$. For example, giving such a homomorphism with $P=\NN$ is equivalent to choosing a pseudo-uniformizer $\pi$. In this situation we use the usual notation $\kcirc_P[Q]=\kcirc\otimes_{\bbZ[P]}\bbZ[Q]$ and also denote its $\varpi$-adic completion (where $\varpi$ is a pseudo-uniformizer) by $\kcirc_P\{Q\}$. The homomorphism $Q\to\kcirc_P\{Q\}$ will be denoted $q\mapsto u^q$.

In this paper, a {\em model formal scheme} is $\gtS_P\{Q\}:=\Spf(\kcirc_P\{Q\})$, where $\lam\:P\into Q$, $\pi\:P\into\kcirc$ are as above and, in addition, the torsion of $Q^\gp/P^\gp$ is of order invertible in $\tilk$ and $\lam$ does not factor through a facet of $Q$. The latter condition is equivalent to the following ones: a) The image of $P_\bfR$ contains a point in the interior of $Q_\bfR$, b) inverting the elements of $P$ one inverts $Q$, that is, $Q+P^\gp=Q^\gp$, c) for any $q\in Q$ there exists $q'\in Q$ such that $q+q'\in\lam(P)$. This condition is satisfied if and only if each $q\in Q$ gives rise to an open ideal $u^q\calO_{\gtS_P\{Q\}}$.

\subsubsection{Examples}\label{twoexam}
The most standard example of a model formal scheme is $\gtS_{\pi,m}$ -- the semistable one. It corresponds to the embedding $\pi\:P=\NN\into\kcirc$, the monoid $Q=\NN^{m+1}$ with generators $e_i$ mapped by $u$ to $t_i$ and the homomorphism given by $\lam(1)=e_0+\dots+e_m$. More general semistable charts $\gtS_{\pi,m}\times\bbG_m^n$ are obtained for $Q=\NN^{m+1}\oplus\ZZ^n$ and the same relation with $P$. They are needed to construct \'etale charts rather than smooth ones, but we will not use them in the paper.

More generally, a standard polystable model $\gtS=\prod_{i=1}^r\gtS_i$, which is a product of standard semistable models with monoids $Q_i$ and $P_i$, corresponds to the monoids $Q=\oplus_{i=1}^rQ_i$ and $P=\sum_{i=1}^rP_i\subset\kcirc$. A formal $\kcirc$-scheme $\gtX$ is {\em strictly polystable} (resp. {\em polystable}) at $x$ if locally  (resp. \'etale locally) at $x$ it possess a smooth morphism to a standard polystable model.

Finally, we will also need the case of certain specific monoids that arise when one blows up semistable models. Consider the semistable scheme
$$\gtX=\Spf\left(\kcirc\{t_0\.t_m,v_1\.v_r\}/(t_0\ldots t_m-\pi)\right),$$ let $v=v_1^l\dots v_r^l$ for some $l>0$ and let $\gtT_{\pi,m,r,l}$ be the $v$-chart of the blowing up of $\gtX$ along $(t_0\.t_{m},v)$. We refer to \cite[Section 2]{BL} for basics on admissible blowings up (a generalization to arbitrary formal blowings up can be found in \cite[p. 495]{temdes}).

\begin{lem}\label{toriclem}
Keep the above notation. Then $$\gtT_{\pi,m,r,l}=\Spf(\kcirc\{t'_0\.t'_{m},v_1\.v_r\}/(t'_0\dots t'_mv_1^d\dots v_r^d-\pi)),$$ where $d=(m+1)l$. In particular, $\gtT_{\pi,m,r,l}=\gtS_P\{Q\}$, where $P=\NN$ is embedded into $\kcirc$ via $\pi$, $Q=\NN^{m+r+1}$ with generators $e_0\.e_{m+r}$ mapped by $u$ to $\ut',\uv$ and $\lam$ is given by $\lam(1)=e_0+\dots+e_m+d(e_{m+n+1}+\dots+e_{m+r})$.
\end{lem}
\begin{proof}
The formal blowing up is obtained by completing the usual blowing up of the corresponding affine scheme $X=\Spec(A)$, where $$A=\kcirc[t_0\.t_{m},v_1\.v_r]/(t_0\dots t_m-\pi),$$ hence it suffices to prove the analogous claim for the blowing up of $X$ along $(t_0\.t_{m},v)$. The usual description of blowings up implies that the $v$-chart equals $\Spec(A')$, where $A'$ is the quotient of $$A'':=A[t'_0\.t'_{m}]/(vt'_0-t_0\.vt'_{m}-t_{m})$$ by the $v$-torsion. Set $f=t'_0\dots t'_mu_1^d\dots u_r^d-\pi$, then $$A''=\kcirc[t'_0\.t'_{m},v_1\.v_r]/(f)$$ and the presence of the free term $\pi$ in $f$ implies that if $vg\in(f)$, then $g\in (f)$. Therefore the $v$-torsion is trivial and the chart is $\Spec(A'')$ as asserted.
\end{proof}

\subsubsection{Semistable modification}
The (non-normal) toric formal $\kcirc$-schemes $\gtT_{\pi,m,r,l}$ can be modified to semistable ones by usual combinatorial tools.

\begin{lem}\label{modiflem}
Let $\gtT=\gtS_P\{Q\}$ be a model formal $\kcirc$-scheme and assume that $P=\NN$ is mapped to $\kcirc$ via a pseudo-uniformizer $\pi$ and $Q^\gp/P^\gp$ is torsion free. Then there exists an extension $l=k(\pi^{1/d})$ and a blowing up $\gtT'\to\gtT\otimes_{\kcirc}\lcirc$ such that $\gtT'$ is strictly semistable over $\lcirc$.
\end{lem}
\begin{proof}
Replacing $\gtT$ by its finite modification $\Spf(\kcirc_P\{Q^\sat\})$ (which is automatically an admissible blowing up) we reduce the claim to the classical case when $Q$ is a toric monoid, and then the argument is standard: first we solve an analogous problem for the $\ZZ[\pi]$-scheme $\Spec(\ZZ[Q])$ and then pull back this solution to a modification of $\gtT\otimes_{\kcirc}\lcirc$ for an appropriate $d$. The main result of \cite{KKMS} provides a solution of the later problem over a field, but the solution over $\ZZ$ is the same: consider the polyhedral cone $\sigma=(Q_\bfR)^{\rm v}$, cut it with the hyperplane $\pi=1$ obtaining a polytope with vertexes in the lattice $(Q^\gp)^{\rm v}$, find $d$ and a unimodular projective subdivision of $\sigma$ in the lattice $d^{-1}(Q^\gp)^{\rm v}$ via the main combinatorial result of \cite{KKMS}, consider the corresponding toric modification of $\Spec(\ZZ[Q]\otimes_{\ZZ[\pi]}\ZZ[\pi^{1/d}])$.
\end{proof}

\begin{rem}
If $P$ is of a larger rank, then a semistable subdivision of $\gtS_P\{Q\}$ in general does not exist, but one can find a polystable subdivision. Essentially, this is the main result of \cite{ALPT}.
\end{rem}

\subsubsection{Log smoothness}
An admissible formal $\kcirc$-scheme is {\em log smooth} (resp. Zariski log smooth) if \'etale locally (resp. locally) it possesses a smooth morphism $f\:\gtU\to\gtS_P\{Q\}$ to a model formal scheme. One can require that $f$ is \'etale, but this leads to the same notion. Note that giving a morphism $f$ is equivalent to giving a $P$-homomorphism of monoids $Q\to\calO_\gtU$, and we call the latter a {\em monoidal chart}. It induces a log structure on $\gtU$ which makes $\gtU$ into a log scheme which is log smooth over $\gtS$ provided with the log structure induced by $P$. Our assumption that $Q^\gp=Q+P^\gp$ just means that the induced log structure on $\gtX_\eta$ is trivial (or the log structure is vertical).

\subsubsection{Log smooth modification conjecture}
Now we can formulate the main global conjecture about existence of nice formal models.

\begin{conj}\label{logsmoothconj}
Assume $k$ is a complete real-valued field and $\gtX$ is an admissible formal $\kcirc$-scheme whose generic fiber $\gtX_\eta$ is rig-smooth. Then there exists a finite extension $l/k$ and an admissible blowing up $\gtX'\to\gtX_l$ such that $\gtX'$ is polystable.
\end{conj}

\begin{rem}
(0) Classically one formulates a weaker conjecture about existence of a single semistable (or log smooth) model of a rig-smooth $X=\gtX_\eta$, and it is called a semistable (or log smooth) reduction conjecture. The modification conjecture essentially asserts that such models form a cofinal family.

(i) It suffices to prove that after a ground field extension $l/k$ there exists a blowing up $\gtX'$ of $\gtX_l$ which is log smooth, because by the main result of \cite{ALPT} such $\gtX'$ can be refined by a log blowing up to a polystable $\gtX''$. If $|k^\times|$ is of rank one, then $\gtX'$ even admits a semistable log blowing up.

(ii) Assume that $\cha(\tilk)=0$. If the valuation is discrete, then the semistable reduction theorem is the famous classical result proved in \cite{KKMS}. The log smooth modification theorem without restrictions on $k$ was established very recently in \cite[Theorem~1.2.19]{ATW-relative} by new resolution techniques.

(iii) If $\dim(\gtX_\eta)=1$, then the analytic semistable reduction theorem is due to Bosch-L\"utkebohmert, see \cite{BL} and the semistable modification follows easily (an alternative proof of semistable modification can be found in \cite{temst}). Nearly nothing is known when $\cha(\tilk)>0$ and $\dim(\gtX_\eta)>1$.
\end{rem}

\subsection{A local criterion of log smoothness}

\subsubsection{Strictly unramified morphisms}
We say that a morphism of schemes $f\:Y\to X$ is {\em strictly unramified} (resp. {\em strictly \'etale}) at $y\in Y$ if it is unramified (resp. \'etale) at $y$ and $k(f(y))=k(y)$. This is a relatively standard terminology for \'etale morphisms, but perhaps not for the unramified ones. A local homomorphism of local rings is called strictly unramified (resp. strictly \'etale) if the induced morphism of spectra is strictly unramified (resp. strictly \'etale) at the closed point of the source. Also, we say that a ring $A$ is {\em generalized Artin} if $\Spec(A)$ is a point. Clearly, this happens if and only if $A$ is a local ring and $m_A$ is the nilradical of $A$.

\begin{lem}\label{artinlem}
Let $\phi\:A\to B$ be a homomorphism of generalized Artin rings, then

(i) $\phi$ is finite if and only if $B$ is a finitely generated $A$-algebra.

(ii) $\phi$ is strictly unramified (resp. strictly \'etale) if and only if $\phi$ is surjective (resp. an isomorphism).
\end{lem}
\begin{proof}
Let $k_A$ and $k_B$ denote the residue fields.

(i) Only the inverse implication needs a proof. If $B$ is finitely generated over $A$, then $k_B$ is a finitely generated $k_A$-algebra and hence $k_B/k_A$ is finite by Hilbert Nullstellensatz. Therefore for any $b\in B$ there exists a monic polynomial $f\in A[t]$ such that the image of $f(b)$ in $k_B$ vanishes. Thus, $f(b)$ is nilpotent, and $f^n$ with a large enough $n$ is a monic polynomial which annihilates $b$. So, $A\to B$ is integral, and hence finite.

(ii) Only the direct implications need a proof. If $\phi$ is strictly unramified, then $\phi$ is finite by (i), $m_AB=m_B$ and the homomorphism $k_A=A/m_A\to B/m_AB=k_B$ is surjective. Therefore $\phi$ is surjective by Nakayama's lemma. If, moreover, $\phi$ is strictly \'etale, then it is also flat, and hence injective.
\end{proof}

\subsubsection{Formal completions}
Given a closed point $x\in\gtX$ we provide the local ring $\calO_{\gtX,x}$ with the adic topology such that a {\em finitely generated} ideal $I\subset\calO_{\gtX,x}$ is an ideal of definition if and only if $I$ is $\pi$-adically open (i.e. contains a power of $\pi$) and the radical of $I/\kcirccirc I$ in $\calO_{\gtX_s,x}$ coincides with $m_{\gtX_s,x}$. For example, a finitely generated open ideal $\cI\subset\calO_{\gtX}$ induces an ideal of definition at $x$ if and only if $V_\gtX(I)=x$. It is easy to see that, indeed, all such ideals define the same adic topology and we denote the completion by $\hatcalO_{\gtX,x}$. Note that $\hatcalO_{\gtX,x}\otimes\tilk=\hatcalO_{\gtX_s,x}$.

\begin{theor}\label{etaleth}
Let $f\:\gtY\to\gtX$ be a morphism of admissible formal $\kcirc$-schemes, $y\in\gtY$ and $x=f(y)$. Then $f$ is strictly unramified (resp. strictly \'etale) at $y$ if and only if the induced homomorphism $\phi\:\hatcalO_{\gtX,x}\to\hatcalO_{\gtY,y}$ is surjective (resp. an isomorphism).
\end{theor}
\begin{proof}
If $\phi$ is surjective (resp. an isomorphism), then tensoring with $\tilk$ we obtain that the homomorphism of completed local rings of $\tilk$-varieties $\hatcalO_{\gtX_s,x}\to\hatcalO_{\gtY_s,y}$ is surjective (resp. an isomorphism). By the classical theory of varieties, this implies that the closed fiber $f_s\:\gtY_s\to\gtX_s$ is strictly unramified (resp. \'etale) at $y$, hence $f$ is strictly unramified (resp. \'etale) at $y$ by Lemma~\ref{simpleetale}.

Conversely, assume that $f$ is strictly unramified (resp. \'etale) at $y$. Then $m_{\gtX_s,x}\calO_{\gtY_s,y}=m_{\gtX_s,y}$ and hence for an ideal of definition $J$ of $\calO_{\gtX,x}$ one also has that $J\calO_{\gtY,y}$ is an ideal of definition of $\calO_{\gtY,y}$. Since $J$ is open it contains an ideal $I=(\pi)$ with a uniformizer $\pi$. By definition, $\phi_I\:\calO_{\gtX,x}/I\to\calO_{\gtY,y}/I\calO_{\gtY,y}$ is a strictly unramified (resp. \'etale) homomorphism of rings, hence also the base change $\phi_J\:\calO_{\gtX,x}/J\to\calO_{\gtY,y}/J\calO_{\gtY,y}$ is strictly unramified (resp. \'etale). By Lemma~\ref{artinlem} $\phi_J$ is surjective (resp. an isomorphism), and passing to the limit over the set of ideals of definition we obtain that $\phi$ is also surjective (resp. an isomorphism).
\end{proof}

\subsubsection{Analytic fibers}
We will also use a relation between the completed local rings of $\gtX$ and domains in the generic fiber. For a closed point $x\in\gtX$ let $\gtX_{\eta,x}$ denote the preimage of $x$ under the reduction (or specialization) map $\gtX_\eta\to\gtX$; it is an open subspace of $\gtX_\eta$. In the framework of rigid geometry the spaces $\gtX_{\eta,x}$ were called {\em formal fibers}, and their various properties were established by Bosch in \cite{Bosch-formalfibers}. However, we prefer to change the terminology and call $\gtX_{\eta,x}$ the {\em analytic fiber} over $x$. 

\begin{lem}\label{fiberlem}
Let $\gtf\:\gtY\to\gtX$ be a morphism of admissible formal $\kcirc$-schemes with generic fiber $f=\gtf_\eta\:Y\to X$, and let $y\in\gtY$ be a closed point and $x=\gtf(y)$. Then the analytic fiber $f_y\:Y_y\to X_x$ depends only on the homomorphism $\phi\:\hatcalO_{\gtX,x}\to\hatcalO_{\gtY,y}$ induced by $\gtf$, and if $\phi$ is surjective (resp. bijective), then $f_y$ is a closed immersion (resp. an isomorphism).
\end{lem}
\begin{proof}
One can assume that $\gtX=\Spf(A)$ and $\gtY=\Spf(B)$. First, let us show how to reconstruct $X_x$ from $\hatcalO_{\gtX,x}$ (see also the proof of \cite[Lemma~4.4]{bercontr}). Fix any $\uh=(h_1\.h_r)\subset A$ such that $x=V_\gtX(\uh)$ and $0\neq\pi\in\kcirccirc$. In particular, $(\uh,\pi)$ is an ideal of definition of $\hatcalO_{\gtX,x}$ and the latter is the completion $\hatA=\hatA_{(\uh,\pi)}$. Note that $X_x$ is given by the inequalities $|h_i|<1$, hence it is the union of rational domains $X_n=X\{\frac{h_1^n}{\pi}\.\frac{h_r^n}\pi\}$. Note that $X_n$ is the generic fiber of the $\pi$-chart $\gtX_n$ of the admissible blowing up of $\gtX$ along $(h_1^n\.h_r^n,\pi)$. In particular, it possesses the standard description $\gtX_n=\Spf(A_n)$, where $A_n$ is the quotient of $$A\{T_1\.T_r\}/(h_1^n-\pi T_1\.h_r^n-\pi T_r)$$ by the $\pi$-torsion. The key observation now is that $\gtX_n$ is also the chart of the blowing up of $\Spf(\hatcalO_{\gtX,x})$ along the same ideal. Indeed, the homomorphism $A\to A_n$ factors through $\hatA=A\llbracket S_1\.S_r\rrbracket/(h_1-S_1\.h_r-S_r)$ and it follows easily that $A_n$ is also the quotient of $$\hatA\{T_1\.T_r\}/(h_1^n-\pi T_1\.h_r^n-\pi T_r)$$ by the $\pi$-torsion. This provides a description of $X_x=\cup_n X_n$ in terms of $\hatA$ and $\uh,\pi$ only. Moreover, a posteriori we know that $X_x$ (but not $X_n$'s) is independent of the choice of $\uh$ and $\pi$, so it is determined by $\hatA$ only.

In the same way one reconstructs $f_y$ from the homomorphism $\phi\:\hatA\to\hatB=\hatcalO_{\gtY,y}$. Choose $\uh,\pi$ as above and let $\ug=(g_1\.g_s)\subset B$ be any family containing $\phi(\uh)$ and such that $V_\gtY(\ug)=y$. Then $Y_y=\cup_nY_n$, where $Y_n=Y\{\frac{g_1^n}{\pi}\.\frac{g_s^n}\pi\}$ is the generic fiber of $\gtY_n=\Spf(B_n)$ with $B_n$ the quotient of $$\hatB\{T_1\.T_s\}/(g_1^n-\pi T_1\.g_s^n-\pi T_s)$$ by the $\pi$-torsion, and $\gtf$ induces morphisms $\gtf_n\:\gtX_n\to\gtY_n$ whose generic fibers $X_n\to Y_n$ after passing to the union give rise to the morphism $X_x\to Y_y$. In particular, if $\phi$ is an isomorphism, then $f_y$ is.

Finally, if $\phi$ is surjective, then we simply take $\ug=\phi(\uh)$ and the induced homomorphisms $A_n\to B_n$ are easily seen to be surjective. Therefore, $\gtf_n$ are closed immersions, their generic fibers $X_n\to Y_n$ are closed immersions, and hence also $X_x\to Y_y$ is a closed immersion.
\end{proof}

As a side remark we discuss a more subtle property, which will not be used.

\begin{rem}
A converse of Lemma \ref{fiberlem} holds for $\eta$-normal formal schemes, i.e. in the case when $\gtX$ is covered by open affines $\Spf(A)$ such that $A=(A\otimes_{\kcirc}k)^\circ$. Indeed, by \cite[Theorem 5.8]{Bosch-formalfibers} the formal completion can be reconstructed from $\gtX_{\eta,x}$ as $\hatcalO_{\gtX,x}=\calO^\circ_{\gtX_\eta}(\gtX_{\eta,x})$. Clearly, this construction is functorial, that is, $\phi$ is reconstructed from $f_y$, and if $f_y$ is a closed immersion (resp. an isomorphism), then $\phi$ is surjective (resp. bijective).
\end{rem}

\subsubsection{Fiber criterion of log-smoothness}
A classical criterion of log smoothness at a point $x$ of a log variety $X$ over a perfect field $k$ is that the log stratum through $x$ is regular of codimension equal to the rank of $\oM_x$. Furthermore, this is the condition used to define log regularity in general. The following result is its direct analogue for admissible formal schemes. We will only use it when the model is $\gtT_{\pi,m,n,l}$, but the proof in the almost general case is the same. For a monoid $Q$ by $Q_+=Q\setminus Q^\times$ we denote its maximal ideal.

\begin{theor}\label{logsmth}
Let $\gtX$ be an admissible formal $\kcirc$-scheme, $P\subseteq Q$ fine sharp monoids with a torsion free $Q^\gp/P^\gp$ and $\pi\:P\into\kcirc$, $u\:Q\into\Gamma(\calO_\gtX)$ compatible homomorphisms. Assume that the closed subscheme $Z=V_{\gtX_s}(u^{Q_+})$ is $\tilk$-smooth at a point $x$ and $$\codim_x(Z,\gtX_s)\ge\rk(Q^\gp/P^\gp).$$ Then the induced morphism $\phi\:\gtX\to\gtT=\Spf(\kcirc_{P}\{Q\})$ is smooth at $x$, the above inequality is an equality and $u^Q\cap\kcirc=\pi^P$.
\end{theor}
\begin{proof}
The last two claims are clear. Indeed, choose $R\subseteq Q$ such that $u^R=u^Q\cap\kcirc$, then $$\codim_x(Z,\gtX_s)\le\rk(Q^\gp/R^\gp)\le\rk(Q^\gp/P^\gp),$$ and hence all inequalities are equalities. Since $Q^\gp/P^\gp$ is torsion free this also implies that $P=R$. It remains to prove the main assertion that $\phi$ is smooth at $x$.

We will use two simple facts about $\gtT$: 1) $\gtT_s$ is equidimensional of dimension $r:=\rk(Q^\gp/P^\gp)$, 2) the analytic fiber $\gtT_{\eta,O}$ over the origin $O=V_{\gtT_s}(Q_+)$ is reduced and irreducible. For example, both claims can be obtained by appropriate base changes from the fact that the morphism $\Spec(\ZZ[Q])\to\Spec(\ZZ[P])$ has equidimensional fibers of dimension $r$ and its fibers over $\Spec(\ZZ[P^\gp])$ are geometrically integral because $Q^\gp/P^\gp$ is torsion free.

Let $l/k$ be the completed maximal unramified extension. All assumptions of the theorem are satisfied for $\gtX_l=\gtX\wtimes_{\kcirc}\lcirc$ and a preimage $x_l$ of $x$. By Lemma~\ref{simpleetale} smoothness of $\phi$ at $x$ is equivalent to smoothness of the closed fiber $\phi_s$ at $x$. By flat decent $\phi_s$ is smooth at $x$ if and only if $\phi_s\otimes_\tilk\tilk^s=(\phi\wtimes_{\kcirc}\lcirc)_s$ is smooth at $x_l$. Therefore it suffices to prove the theorem for $\gtX_l$ and $x_l$, and replacing $l$ by $k$ for shortness, we can assume in the sequel that $\tilk$ is separably closed.

The assertion of the theorem is local at $x$. Furthermore, let $y$ be any closed point belonging to the closure of $x$ and such that $Z$ is $\tilk$-smooth at $y$ and $\codim_x(Z,\gtX_s)=\codim_y(Z,\gtX_s)$. Then $y$ satisfies all assumptions of the theorem, and it suffices to prove that $\phi$ is smooth at $y$. Thus, replacing $x$ by $y$ we can assume in the sequel that $x$ is a closed point.

Set $d=\dim(\calO_{Z,x})$ and choose elements $\tilt_1\.\tilt_d\in\calO_{Z,x}$ such that the morphism $\tilt\:Z_0\to\bfA^d_\tilk$ they induce on a neighborhood $Z_0\subseteq Z$ of $x$ is \'etale at $x$. Shrinking $\gtX$ around $x$ we can assume that these elements are images of global functions $t_1\.t_d\in\Gamma(\calO_\gtX)$, and hence a morphism $\psi=(\phi,t)\:\gtX\to\gtU=\gtT\times\bfA_{\kcirc}^d$ arises. We claim that $\psi$ is \'etale at $x$. Since $\phi$ is obtained by composing $\psi$ with the projection onto $\gtT$, this will immediately imply the assertion of the theorem.

We start with studying the closed $\psi_s$. Note that $u:=\psi(x)=(O,u_0)$, where $O=V_{\gtT_s}(u^{Q_+})$ is the origin of $\gtT_s$ and $u_0=\tilt(x)$. In addition, $k(O)=\tilk$ and since $k(u_0)$ is finite over the separably closed field $\tilk$ and $k(x)/k(u_0)$ is separable, we also have that $k(x)=k(u_0)=k(u)$. We claim that $\psi_s$ is strictly unramified at $x$, that is, $x=\Spec(k(x))$ is an isolated component of the fiber of $\psi_s\:\gtX_s\to\gtU_s$ over $u$. To describe this fiber we can replace $\psi_s$ by its base change with respect to the closed immersion $O\times\bfA_{\tilk}^d\into\gtU_s$, but the latter base change coincides with $\tilt$, so it is strictly \'etale at $x$ and the fiber is as claimed.

Since $\psi_s$ is strictly unramified at $x$, Lemma~\ref{artinlem} implies that the corresponding homomorphism of noetherian complete local rings $h\:\hatcalO_{\gtU_s,u}\to\hatcalO_{\gtX_s,x}$ is surjective. To complete the proof it suffices to show that $h$ is also injective, because then $\gtX_s\to\gtU_s$ is \'etale at $x$ by the classical theory and we can conclude by Lemma~\ref{simpleetale}. However, the involved rings are not domains, hence the dimension considerations alone are insufficient. Instead of this we will now lift this homomorphism to a homomorphism of domains.

By Theorem~\ref{etaleth} the surjectivity of $h$ implies that the homomorphism $h'\:\hatcalO_{\gtX,x}\to\hatcalO_{\gtU,u}$ is also surjective and hence the associated morphism of analytic fibers $\gtX_{\eta,x}\to\gtU_{\eta,u}$ is a closed immersion by Lemma~\ref{fiberlem}. The source and the target are $k$-analytic spaces of dimension $d+r$ and the target is reduced and irreducible because it is the product of reduced and irreducible spaces $\gtT_{\eta,O}$ and $(\bfA^d_{\kcirc})_{\eta,u_0}$. Therefore this closed immersion is an isomorphism, and hence does not factor through the vanishing locus of any non-zero element of $\hatcalO_{\gtX,x}$. Thus, $\Ker(h')=0$ and we obtain that $h'$ is an isomorphism. Therefore its closed fiber $h$ is an isomorphism too, concluding the proof.
\end{proof}

\subsubsection{Reduction to the case of an algebraically closed ground field}
Using the above theorem we can now descend log smoothness from $\gtX_{(\whka)^\circ}$. This result will not be used so the reader can skip it. The proof is nearly the same as in the semistable case.

\begin{theor}\label{algcloslem}
Let $\gtX$ be an admissible formal $\kcirc$-scheme, $K=\whka$ and $\gtX_{\Kcirc}=\gtX\wtimes_{\kcirc}\Kcirc$. Assume that $f_{\Kcirc}\:\gtX'_{\Kcirc}\to\gtX_{\Kcirc}$ is an admissible formal blowing up such that $\gtX'_{\Kcirc}$ is Zariski log smooth at a point $x_K$. Then there exists a finite separable extension $l/k$ and an admissible formal blowing up $f_{\lcirc}\:\gtX'_{\lcirc}\to\gtX_{\lcirc}$ such that $f_{\Kcirc}$ is the pullback of $f_{\lcirc}$ and $\gtX'_{\lcirc}$ is Zariski log smooth at the image $x_l\in\gtX'_{\lcirc}$ of $x_K$.
\end{theor}
\begin{proof}
Fix a smooth morphism $g_{\Kcirc}:\gtU_{\Kcirc}\to\Spf(\Kcirc_P\{Q\})$ from a small enough affine neighborhood $\gtU_{\Kcirc}=\Spf(A_{\Kcirc})$ of $x_K$ in $\gtX'_{\Kcirc}$ and let $\phi_{\Kcirc}\:Q\to\calO_{\gtU_{\Kcirc}}$ be the associated monoidal chart. The same argument as in the proof of Corollary~\ref{semistabledescent} shows that for a large enough $l$ the blowing up $f_{\Kcirc}$ is obtained from a blowing up $f_{\lcirc}\:\gtX'_{\lcirc}\to\gtX_{\lcirc}$ and $\gtU_{\Kcirc}$ is the preimeage of a neighborhood $\gtU_{\lcirc}=\Spf(A_{\lcirc})$ of $x_l$. We claim that enlarging $l$ if needed one can find a monoidal chart $\phi'_{\lcirc}\:Q\to\calO_{\gtU_{\lcirc}}$ such that for any $q\in Q$ we have that $\phi_{\Kcirc}(q)=u_q\phi'_{\lcirc}(q)$ in $\calO_{\gtU_{\Kcirc}}$ with an invertible $u_q$. This will finish the proof because then $\phi'_{\lcirc}$ induces a chart $g'_{\lcirc}\:\gtU_{\lcirc}\to\Spf(\lcirc_P\{Q\})$ whose base change $g'_{\Kcirc}\:\gtU_{\Kcirc}\to\Spf(\Kcirc_P\{Q\})$ is smooth at $x_K$ by Theorem~\ref{logsmth}, since $g'_{\Kcirc}$ and $g_{\Kcirc}$ have the same fiber over the origin.

To construct $\phi'_{\lcirc}$ choose $q_1\.q_d\in Q$ whose images form a basis of $Q^\gp/P^\gp$. Each $q_i$ divides some element of $P$, hence they all divide some $p\in P$ and then each $u^{q_i}$ divides $\pi^p$ in $\Kcirc_P\{Q\}$. The same argument as in the proof of Corollary~\ref{semistabledescent} shows that enlarging $l$ we can find units $u_i\in 1+\Kcirccirc A_{\Kcirc}$ such that $a_i=u_i\phi(q_i)\in A_{\lcirc}$. We claim that there exists a unique $P$-homomorphism $\phi'\:Q\to A_{\lcirc}$ such that $\phi'(q_i)=a_i$. Each $q\in Q$ can be presented as $q=p+\sum_{i=1}^d l_iq_i$ with $l_i\in\bbZ$ and $p\in P^\gp\subset k^\times$, so using that $\phi(q)=\pi^p\prod_{i=1}^d\phi(q_i)^{l_i}$, we obtain that the rule $\phi'(q)=\phi(q)\prod_{i=1}^d u_i^{l_i}\in A_{\Kcirc}$ defines a $P$-homomorphism $\phi'\:Q\to A_{\Kcirc}$. It remains to show that the image is actually in $A_{\lcirc}$. Choose a representation $q=q'-q''$, where $$q'=p'+\sum_{i=1}^dl'_iq_i,\  q''=p''+\sum_{i=1}^dl''_iq_i,\  p,p'\in P,\ l'_i,l''_i\in\bbN.$$ Then $\phi'(q')=\pi^{p'}\prod_{i=1}^da_i^{l'_i}\in A_{\lcirc}$, in the same way $\phi(q'')\in A_{\lcirc}$, and we have that $\phi'(q)=\phi'(q')/\phi'(q'')$ in $A_{\Kcirc}$. Then the same argument as in Corollary~\ref{semistabledescent} shows that a quotient exists already in $A_{\lcirc}$ and hence $\phi'(q)\in A_{\lcirc}$.
\end{proof}

As an immediate corollary we obtain that it suffices to verify the log smooth modification conjectures in the case, when the ground field $k$ is algebraically closed. We already observed the similar fact for the local uniformization, but include it in the following formulation too:

\begin{cor}\label{algcloscor}
If Conjectures \ref{locunifconj} and \ref{logsmoothconj} hold true when the ground field $k$ is algebraically closed, then they hold in general.
\end{cor}

\subsection{Reduction to analytic points}\label{mainsec}

\subsubsection{Full family of parameters}
Assume that $x$ is a strictly semistable point of an admissible formal $\kcirc$-scheme $\gtX$. By a {\em (twisted) family of parameters} $(\ut,\us)$ at $x$ we mean (twisted) semistable parameters $\ut=(t_0\.t_m)$ at $x$ and a tuple of elements $\us=(s_1\.s_n)\in\cO_{\gtX,x}$ such that if $V=V_{\gtX_s}(t_0\.t_m)$, then the image of $\us$ in the regular ring $\cO_{V,x}$ is a family of regular parameters. If $\us$ is only a partial family of regular parameters, then we say that $(\ut,\us)$ is a (twisted) partial family.



\subsubsection{Key lemma}
The following analogue of Lemma~\ref{keylem} will be our key lemma in the study of analytic local uniformization. Our argument is a modification of the proof of that lemma, though it gets more complicated when $\gtX$ is not smooth but only semistable at $x$, as some log geometry is naturally involved. In particular, we will use the log smooth models $\gtT$ introduced in \S\ref{twoexam}.

\begin{lem}\label{formalkey}
Assume that $\tilk$ is perfect, $\gtX$ is an admissible formal $\kcirc$-scheme and $Y\into\gtX_s$ an integral closed subscheme with generic point $\eta$. Assume that $x\in Y$ is a point and $E\subset Y$ a divisor such that $(Y,E)$ is a regular pair at $x$ and $\gtX$ is strictly semistable at any point of $Y\setminus E$. Finally, assume that there exist elements $\ut=(t_0\.t_m)$, $\us=(s_1\.s_n)$ of $\cO_{\gtX,x}$ such that their images form a partial twisted family of parameters at any point $y\in Y\setminus E$ and a full family of parameters at $\eta$. Then there exists an admissible blowing up $f\:\gtX'\to\gtX$ such that the strict transform $g\:Y'\to Y$ is an isomorphism over $x$ and there exist a neighborhood $\gtU\subseteq\gtX'$ of $x'=g^{-1}(x)$ and a smooth morphism $\psi\:\gtU\to\gtT_{\pi,m,r,l}$ for $r=\dim(\calO_{Y,x})$, a pseudo-uniformizer $\pi$ and some $l>0$.
\end{lem}
\begin{proof}
Enlarging $E$ we can assume that the number of its irreducible components at $x$ equals $r$. Recall that admissible blowings up can be extended from an open formal subscheme by \cite[Lemma~2.6(a)]{BL}. Therefore the assertion of the lemma is local at $x$, and shrinking $\gtX$ we can assume that $\gtX=\Spf(A)$ and $Y=\Spec(B)$ are affine, $\ut,\us\subset A$, the pair $(Y,E)$ is regular, and there exist elements $\uv=(v_1\.v_r)$ in $A$ with $v=v_1\dots v_r$ and images $\uw=(w_1\.w_r)$ and $w=w_1\dots w_r$ in $B$ such that $E=V_Y(w)$. In particular, the image of $\uv$ in $\cO_{Y,x}$ is a family of regular parameters and $\gtX$ is semistable away from $V_\gtX(v)$.

Set $C=A/(\kcirccirc,\ut,\us)$ and $Z=\Spec(C)$. Then $B=C/J$ for a prime ideal $J$ and $Y\setminus E\into Z$ is an open immersion because it is of codimension 0 and $Z$ is integral (even regular) at any point of $Y\setminus E$. Therefore we obtain by Lemma~\ref{annlem} that $J=\Ann(c)$, where $c$ is a lift of $w^\ell$ to $C$ and $\ell$ is large enough. Let $a\in A$ be a lift of $c$ and let $f\:\gtX'\to\gtX$ be the blowing up along $(\ut,\us,a^2)$. We will see that $f$ is as required.

The $a^2$-chart $\gtX'_{a^2}=\Spf(A')$ is the $\pi$-adic completion of the analogous chart of the blowing up of $\Spec(A)$, so $A'$ is the completion of $A[\ut',\us']\subseteq A_a$, where $t'_i=t_i/a^2$ and $s'_j=s_j/a^2$. We will use the notation $A'=A\{\ut',\us'\}$ to denote that $A'$ is topologically generated by $\ut',\us'$. Applying claims (i) and (iii) of Lemma~\ref{strictlem} to the blowing up of $\Spec(A)$ along $(\ut,\us,a^2)$ and passing to the formal completion (which does not modify the closed fiber) we obtain the same conclusions for $\gtX'$: the strict transform $Z'\into\gtX'$ of $Z$ is isomorphic to $Y$, contained in $\gtX'_{a^2}$ and given by vanishing of the ideal $(\kcirccirc,\ut',\us')$.

In addition, we claim that $a=uv^\ell$ for an element $u\in A'$ which is a unit in a neighborhood of $x'$. Consider the blowing up $\gtX''\to\gtX$ along $(\ut,\us,a)$ with the $a$-chart $\gtX''_a=\Spf(A'')$, where $A''=A\{\ut'',\us''\}$. By the same application of Lemma~\ref{strictlem} as above, the strict transform of $Z$ is isomorphic to $Y$ and given by the vanishing of $(\ut'',\us'')$. Since the images of $a$ and $v^\ell$ in $A''/(\ut'',\us'')=B$ coincide and $\gtX'_a$ is the $a$-chart of the blowing up of $\gtX''_a$ along $(\ut'',\us'',a)$, Lemma~\ref{strictlem}(ii) implies that, indeed, $a=uv^{\ell}$ in $A'$ and $u$ is invertible along $Z'$.

Since $t_i=a^2t'_i$, we have that $a^dt'_0\dots t'_m=t_0\dots t_m=\pi$, where $d=2l(m+1)$. Setting $y_0=u^dt'_0$ and $y_i=t'_i$ for $1\le i\le m$, we achieve that $y_0\dots y_m v^d=\pi$. Set $Q=\NN^{m+r+1}$ and let $\lam\:P=\NN\to Q$ be the homomorphism defined by $$\lam(1)=e_0+\dots+e_m+d(e_{m+1}+\dots+e_r).$$ Then the exponential homomorphism $u\:Q=\NN^{m+r+1}\to A$ sending the basis elements to $y_0\.y_m,v_1\.v_r$ and the homomorphism $\pi:P\into\kcirc$ induced by the pseudo-uniformizer $\pi$ are compatible, and hence give rise to a morphism $\psi\:\gtU:=\gtX'_{a^2}\to\gtS_{P}\{Q\}=\gtT_{\pi,m,r,l}$ (see Lemma~\ref{toriclem}). We will complete the proof by proving that $\psi$ is smooth at $x'$. Furthermore, in view of Theorem~\ref{logsmth} it suffices to prove the following two claims: (i) the closed subscheme $W'=V_{\gtX'_s}(u^{Q_+})$ is $\tilk$-smooth at $x'$, (ii) $\codim_{x'}(W',\gtX'_s)=\rk(Q^\gp/P^\gp)=m+r$.

The closed subscheme $T'=V_{\gtX'_s}(u^{Q_+},\us)=V_{\gtX'_s}(\ut',\us',\uv)$ coincides with $V_{Z'}(\uv)$. Since the image of $\uv$ is a regular family of parameters of $\cO_{Z',x'}$ and $\tilk$ is perfect, we obtain that $x'$ is a generic and $\tilk$-smooth point of $T'$ and $\codim_{x'}(T',Z')=r$. Since $\codim_{x'}(Z',\gtX'_s)=m+n$, we obtain that $\codim(T',\gtX'_s)=m+n+r$. Since $W'=V_{\gtX'_s}(\ut',\uv)$ is of codimension at most $m+r$ at $x'$ and $T'=V_{W'}(\us')$ is of codimension at most $n$ in $W'$, both estimates are equalities, yielding (ii). Since $T'=V_{W'}(\us')$ is $\tilk$-smooth at $x'$ and of codimension $n$ in $W'$, we automatically obtain that $W'$ is $\tilk$-smooth at $x'$, as claimed by (i).
\end{proof}

\subsubsection{The main theorem}
Finally, we are in a position to prove our main result about local uniformization of analytic spaces. It reduces proving the analytic local uniformization conjecture to the case of Berkovich points and the local log uniformization of valuations on $\tilk$-varieties. The proof is somewhat similar to the proof of its algebraic analogue -- Theorem~\ref{unifschemes}.

\begin{theor}\label{mainth}
Assume that $X$ is a smooth strictly $k$-analytic space of dimension $d$ such that any point of $X$ is uniformizable. Assume also that the local log uniformization conjecture holds for $\tilk^a$-varieties of dimension at most $d$. Then any point of $X^\ad$ is uniformizable.
\end{theor}
\begin{proof}
First we note that any point of $X_{\whka}$ is also uniformizable, hence the assumptions are also satisfied for $X_{\whka}$. Moreover, proving the theorem for the latter will imply the theorem for $X$ by Lemma~\ref{unifdescent}. Thus, in the sequel we can assume that $k$ is algebraically closed. Fix a formal model $\gtX$ of $X$ and let $z\in X^\ad$ be a point. We will construct a blowing up of $\gtX$ which uniformizes $z$ by composing a sequence of blowings up which gradually improve the situation at the specialization of $z$. For simplicity of notation, we will replace $\gtX$ with the intermediate blowings up and only record the properties achieved so far. Each next blowing up will be denoted $\gtX'\to\gtX$.

We will use the following notation: $y\in X$ is the generization of $z$ of height 1, $\eta\in\gtX$ is the specialization of $y$, $Y\into\gtX_s$ is the Zariski closure of $\eta$ and $\lam$ is the valuation on $k(Y)$ induced by $z$. In particular, the center $x\in Y\subset\gtX_s$ of $\lam$ is the specialization of $z$.

Step 1. {\it One can achieve that $\eta$ is a strictly semistable point.} Indeed, by assumptions of the theorem blowing up $\gtX$ one can achieve that $y$ specializes to a strictly semistable point $\eta$.

In the sequel we will only use admissible blowings up which induce isomorphisms over $\eta$, so the situation in its neighborhood will stay unchanged and we denote its preimages in any $\gtX'$ by the same letter. Our next goal is to choose parameters at $\eta$ which extend to $x$. Naturally, we will start with any set of parameters and use blowings up to extend them.

Step 2. {\it In addition to condition of step 1 one can achieve that there exist elements $t_0\.t_{m}\in\calO_{\gtX,x}$ whose images in $\calO_{\gtX,\eta}$ are twisted semistable parameters at $\eta$.} By Lemma~\ref{semistlem} there exists a smooth morphism $\gtU\to\gtS_{\pi,m}$ from a neighborhood of $\eta$ which sends $\eta$ to the origin. Let $t'_0\.t'_m\in\Gamma(\calO_\gtU)$ be the corresponding semistable parameters. Each $t'_i$ induces an open invertible ideal $\calI_{\gtU,i}$ on $\gtU$ and we choose any its extension to a finitely generated open ideal $\calI_i$ on $\gtX$ by the argument from the proof of \cite[Lemma~2.6(a)]{BL} (in fact, one reduces the problem to the case of qcqs schemes established in \cite[Th\'eor\`eme~6.9.7]{egaI}). Replacing $\gtX$ by the blowing up along $\prod_{i=0}^m\calI_i$ we do not change the situation above $\eta$ and achieve that each $\calI_i$ is invertible, hence we can choose a generator $t_i\in\calI_i\calO_{\gtX,x}$ at $x$. Since $t_i$ and $t'_i$ define the same ideal at $\eta$, the elements $t_0\.t_m$ are as required. Note that at this stage we cannot get rid of the twist because even though $t_0\dots t_m=u\pi$ and $u$ is invertible at $\eta$, it does not have to be invertible at $x$.

Step 3. {\it In addition to conditions of steps 1 and 2 one can achieve that there exists a tuple $\us=(s_1\.s_n)\subset\calO_{\gtX,x}$ such that the image of $(\ut,\us)$ form a twisted family of parameters at $\eta$.} Choose $\us'=(s'_1\.s'_n)$ such that $(\ut,\us')$ is a twisted family of parameters at $\eta$. Set $W=V_{\gtX_s}(\ut)$, then the images $(\os'_1\.\os'_n)\subset\cO_{W,\eta}$ form a regular family of parameters. To extend them to $x$ we use the trick from step 2: choose ideals $\cI_1\.\cI_n\subset\calO_{W}$ such that $\cI_i=(\os'_i)$ locally at $\eta$ and let $g'\:W''\to W$ be the blowing up along $\cI=\cI_1\dots\cI_n$. Note that $g'$ makes the pullbacks of each $\cI_i$ invertible. In addition, $V_W(\cI)$ is a Cartier divisor at $\eta$, hence $g'$ is an isomorphism over $\eta$.

It will be crucial in the sequel to refine $g'$ by another blowing up so that $\eta$ is not contained in the center. Fortunately, this is possible by a strong Chow lemma asserting that blowings up with centers at a closed subset $V\subset W$ are cofinal among all modifications of $W$ which are isomorphism outside of $V$. In particular, there exists an ideal $\cJ\subset\cO_W$ trivial at $\eta$ and such that the blowing up $g\:W'=\Bl_\cJ(W)\to W$ factors through $W''$. We still have that $g$ is trivial over $\eta$ and principalizes each $\cI_i$. In particular, if $x'$ is the center of $\lam$ on the closure $Y'$ of $\eta$ in $W'$, then choosing generators $\os_i$ of $\cI_i\calO_{W',x'}$ we obtain that $\os_i=u_i\os'_i$ in $\calO_{W',\eta}=\cO_{W,\eta}$ where $u_i$ are units, and hence $\os_1\.\os_n$ form a family of regular parameters of $\cO_{W,\eta}$.

Now let us lift this to $\gtX$. Find any extension of $\cJ$ to an open finitely generated ideal on $\gtX$ which is trivial at $\eta$ and let $f\:\gtX'\to\gtX$ be the blowing up of this ideal. Then $W'\to W$ is the strict transform of $f$, hence $W'$ is a closed subscheme of $\gtX'_s$ and clearly $x'$ is the specialization of $z$ on $\gtX'$. Thus we can take any lift $s_i\in\calO_{\gtX',x'}$ of $\os_i\in\calO_{W',x'}$ and it remains to replace $\gtX$ by $\gtX'$. The condition of step 2 is satisfied because we can just pull back $t_0\.t_m$ to $\cO_{\gtX',x'}$.

Step 4. {\it In addition to conditions of steps 1,2,3 one can achieve that there exists a divisor $E\subset Y$ such that (a) $(Y,E)$ is a regular pair at $x$, (b) for any point $w\in Y\setminus E$ one has that $\gtX$ is strictly semistable at $w$ and $(\ut,\us)$ form a twisted partial family of parameters at $w$.} The second condition is satisfied at $\eta$, hence we can choose a large enough divisor $E\subset Y$ such that this condition is also satisfied at any point of $Y\setminus E$. In particular, $Y$ is regular at any point of $Y\setminus E$. It remains to achieve that also a) is satisfied without destroying b).

Since $\dim(Y)\le\dim(\gtX_s)\le d$ the valuation $\lam$ is log uniformizable by assumptions of the theorem. Therefore there exists a blowing up $g\:Y'=\Bl_V(Y)\to Y$ such that the pair $(Y',E'=g^{-1}(V\cup E))$ is regular at the center $x'\in Y'$ of $\lam$. Choose any lift of the ideal $\cI_V\subseteq\cO_Y$ to an open finitely generated ideal of $\gtX$ and let $g\:\gtX'\to\gtX$ be the corresponding blowing up. Then $Y'\to Y$ is the strict transform of $g$, condition (a) is satisfied because $(Y',E')$ is a regular pair, and condition (b) is satisfied because for any point $w\in Y'\setminus E'$ the morphism $g$ is a local isomorphism at $u$ and maps it to $Y\setminus E$. Thus, replacing $\gtX$, $E$ and $\ut,\us$ by $\gtX'$, $E'$ and the pullbacks of these tuples to $\cO_{\gtX',x'}$ we accomplish the step.

Step 5. {\it End of proof.} The conditions achieved in steps 1,2,3,4 allow us to use the formal key lemma \ref{formalkey}, thereby obtaining a blowing up $\gtX'\to\gtX$ such that the strict transform $g\:Y'\to Y$ is an isomorphism over $x$ and the preimage $x'=g^{-1}(x)$ has a neighborhood $\gtU$ which admits a smooth morphism to a log smooth scheme $\gtT=\gtT_{\pi,m,r,l}$. Note that $x'$ is the center of $\lam$ and hence also the specialization of $y$. Finally, by Lemma~\ref{modiflem} there exists an admissible (even monomial) blowing up $\gtT'\to\gtT$ with a semistable source. Pulling it back we obtain an admissible blowing up $h\:\gtU'\to\gtU$ such that $\gtU'$ is smooth over $\gtT'$ and hence semistable. Clearly, $y$ specializes to a point of $\gtU'$, and it remains to arbitrarily extend $h$ to an admissible blowing up $\gtX'\to\gtX$.
\end{proof}

\bibliographystyle{amsalpha}
\bibliography{induction_on_height}

\end{document}